\documentclass[12pt, leqno]{amsart}

\usepackage{indentfirst}
\usepackage{amstext}
\usepackage{amsthm}
\usepackage{amsopn}
\usepackage{amsfonts}
\usepackage{amsmath}
\usepackage{latexsym}
\usepackage[english]{babel}
\usepackage{amscd}
\usepackage{amssymb}
\usepackage{amsmath}
\usepackage[all,cmtip]{xy}
\usepackage{graphicx}
\usepackage{etoolbox}
\patchcmd{\section}{\normalfont\scshape\centering}{\normalfont\bfseries}{}{}
\patchcmd{\subsection}{-.5em}{.5em}{}{}
\renewenvironment{proof}{{\noindent\bfseries Proof.}}{}
\usepackage{hyperref}

\newtheorem{theo}{{Theorem}}[section]
\newtheorem{coro}[theo]{{Corollary}}
\newtheorem{lemma}[theo]{{Lemma}}
\newtheorem{prop}[theo]{Proposition}

\newtheorem*{claim}{{Claim}}

\theoremstyle{definition}
\newtheorem{remark}[theo]{\textbf{Remark}}
\newtheorem{defn}[theo]{Definition}
\newtheorem{example}[theo]{Example}
\numberwithin{equation}{section}

\newtheorem{notation}[theo]{Notation}
\newtheorem{question}[theo]{Question}

\newcommand{\End}{\operatorname{End}}

\newcommand{\rT}{\mathrm{T}}

\newcommand{\TT}{\mathrm{T}}

\newcommand{\Aut}{\mathrm{Aut}}

\newcommand{\id}{\mathrm{id}}

\newcommand{\disc}{\mathrm{disc}}

\newcommand{\fQ}{{\mathbf{Q}}}

\newcommand{\QQ}{{\bf{Q}}}
\newcommand{\CC}{{\bf C}}
\newcommand{\PP}{{\bf P}}
\newcommand{\RR}{{\bf R}}
\newcommand{\ZZ}{{\bf Z}}

\newcommand{\Pic}{\operatorname{Pic}}

\DeclareFontEncoding{OT2}{}{} 
  \newcommand{\textcyr}[1]{%
    {\fontencoding{OT2}\fontfamily{wncyr}\fontseries{m}\fontshape{n}%
     \selectfont #1}}
\newcommand{\sha}{{\mbox{\textcyr{Sh}}}}

\begin{document}
\tolerance 400 \pretolerance 200 \selectlanguage{english}


\title{Non-projective K3 surfaces with real or Salem multiplication}

\author{Eva  Bayer-Fluckiger}

\address{EPFL-FSB-MATH, Station 8, 1015 Lausanne, Switzerland}
\email{eva.bayer@epfl.ch}

\author{Bert van Geemen}

\address{Dipartimento di Matematica, Universit\`a di Milano,
Via Saldini 50, I-20133 Milano, Italia}
\email{lambertus.vangeemen@unimi.it}

\author{Matthias Sch\"utt}

\address{Institut f\"ur Algebraische Geometrie, Leibniz Universit\"at
  Hannover, Welfengarten 1, 30167 Hannover, Germany\newline
  Riemann Center for Geometry and Physics, Leibniz Universit\"at
  Hannover, Appelstrasse 2, 30167 Hannover, Germany}
\email{schuett@math.uni-hannover.de}

\date{\today}

\begin{abstract} We determine the Hodge endomorphism algebras of non-projective
complex K3 surfaces (and more generally, hyperk\"ahler manifolds). 
We show that they are either totally real fields or number fields generated by Salem
numbers.
This is unlike the projective case, where the endomorphism fields are
either totally real or CM.
We also develop precise existence criteria and explore
the relations to number theory and dynamics.

\end{abstract}

\maketitle

\small{} \normalsize

\selectlanguage{english}
\section{Introduction}

The Hodge endomorphisms of a complex K3 surface are the $\QQ$-linear endomorphisms that preserve
the rational transcendental Hodge structure on the second cohomology group.
It is well-known that the algebra of Hodge endomorphisms of a projective K3 surface 
is a totally real number field or a CM field; this was proved by Zarhin in
\cite {Z}. 
Accordingly, one refers to real multiplication (RM) and complex multiplication (CM).
In \cite{GS} the RM case was investigated and explicit geometric examples were given. In \cite{BGS} we gave a complete classification of the RM and CM fields that occur
as Hodge endomorphism algebras of  projective K3 surfaces and, more generally, of the known 
hyperk\"ahler manifolds (HK manifolds).

\medskip
As we will see, the Hodge endomorphisms show quite different behaviour according to the algebraic dimension of the K3 surface. A complex K3 surface has algebraic dimension 0, 1 or 2.
The latter corresponds to projective surfaces, which we do not consider again.
In case the algebraic dimension is 1, we show in Proposition \ref{prop:hodge 1}
that the Hodge endomorphisms are the scalar multiples of the identity.

\medskip
This paper is thus mainly dedicated to the case that the algebraic dimension is zero.
The algebra of Hodge endomorphisms can then be a totally real field, and it can also be a {\it Salem field}, defined as follows:

\begin{defn}  
\label{def:Salem'}
We say that an algebraic number field  is a {\it Salem  field} if it 
has two real embeddings, the other embeddings being complex, its degree  is $\geqslant 4$, and it has a $\QQ$-linear involution with totally real fixed field.

\end{defn}

We say that the K3 surfaces with algebra of Hodge endomorphisms 
a Salem field  have
{\it Salem  multiplication}, denoted by SM; the reason for this terminology is explained below.

\medskip 
If $X$ is a K3 surface or a   hyperk\"ahler manifold, we denote by $A_X$ its algebra of Hodge endomorphisms,
by $T_X$ its transcendental lattice, and we set $T_{X,\QQ} = T_X \otimes_{\ZZ}{\QQ}$. We start by defining the notion of
{\it pseudo-polarized Hodge structure} of K3 type (see Definition \ref{def:T}). This notion will be central in the whole paper; indeed,  the usual notion of polarized Hodge structure is
not appropriate for our purposes, as we will see in \S \ref{s:K3}.

\medskip

The first main result of the paper is the following

\begin{theo}\label{endos}
Let $X$ be a K3 surface with $a(X) = 0$,
or more generally a hyperk\"ahler manifold $X$ with pseudo-polarized Hodge structure $T_{X,\QQ}$. 
Let $E =A_X$ be the algebra of Hodge endomorphisms of $T_{X,\QQ}$, and set
 $m = {\rm dim}_E(T_{X,\QQ})$.
Then $E$ is a number field with a $\bf Q$-linear involution
and $E$ is either totally real, with trivial involution and $m\geq 3$, or a Salem field, with $m=1$.
%
\end{theo}



The proof of this result, an analog of Zarhin's result \cite{Z} in the non-projective case, 
is the subject matter of the first part of the paper, comprising the first five sections.
It builds on the interplay of Hodge theory, quadratic and hermitian forms, 
and in particular their {\it signatures}.

\medskip
This first result raises the question to decide which fields and which ranks of transcendental lattices actually occur; in the projective case, this
is answered in \cite{BGS}. 
The second part of the paper,
from Section \ref{s:period} to \ref{s:HK}, is devoted to the answer to this question in the {\it real multiplication} case;
for obvious reasons, we only consider RM by fields different from $\QQ$. 
For simplicity,
we give the result in the case of K3 surfaces of algebraic dimension 0; similar results are obtained in the more general case of HK manifolds.

\begin{theo}\label{RM}
Let $E$ be a totally real number field of degree $d$ and let $m$ be an integer with $3\leq m \leq 22/d$.
Then there exists an $(m-2)$-dimensional family of
K3 surfaces of algebraic dimension $0$ such that a very general member $X$ has the properties
$A_X \simeq E$ and ${\rm dim}_E(T_{X,\QQ}) = m$.
\end{theo}

This is proved in Section  \ref{ss:RM}, and the result concerning hyperk\"ahler manifolds 
can be found in Section \ref{s:HK}. Both require 
arithmetic results concerning quadratic forms over totally real fields and their ``transfers'' (see 
Sections \ref{s:period}
 - \ref{s:HK}). 


\medskip
The third part of the paper (Sections \ref{Salem multiplication section} -- \ref{s:SM}) continues with the results concerning 
\emph{Salem multiplication}. The following is proved in \S \ref{prescribed Salem}, see Theorem \ref{K3}.

\begin{theo}\label{SM} Let $E$ be a Salem field of degree $d$. Then there exists a
K3 surface $X$ of algebraic dimension $0$ such that
$A_X \simeq E$  if and only if either $d \leqslant 20$, or $d = 22$ and
the discriminant of $E$ is a square.
\end{theo}

Note that here we always have ${\rm dim}_E(T_{X,\QQ}) = 1$, as already shown in Theorem \ref{endos}. 
The results concerning higher-dimensional
HK manifolds 
comprise Theorem \ref{non-maximal theorem} and Theorem \ref{610}.
The proofs of Theorem \ref{SM} and its generalizations to HK manifolds occupy the third part of the paper. 
They
rely on Theorem  \ref{thm:pspolSM}, as well as on results concerning the arithmetic of hermitian forms over Salem fields given in  \S \ref{Salem multiplication section} and \S \ref{hermitian BF}.

\medskip
The notion of Salem field is motivated by the fact that it occurs as Hodge endomorphism algebra for some K3 surfaces, as shown by
the above results. However, this notion also appears in other parts of mathematics : complex dynamics and number theory. 
To explain these connections, we start
with some remarks about the terminology ``Salem field'' and ``Salem multiplication''. We first recall the definitions of Salem polynomials and
Salem numbers~:

\begin{defn}\label{definition Salem polynomial}
A {\it Salem polynomial} is a monic irreducible polynomial $S \in {\bf Z}[x]$ of degree $\geqslant 4$ such that
$S(x) = x^{{\rm deg}(S)} S(x^{-1})$ and that $S$ has exactly two roots not on the unit circle, both positive
real numbers. The unique real root $\lambda> 1$ of a Salem polynomial is called a {\it Salem number}.

\end{defn}

We refer to \cite {GH} and \cite{Sm} for surveys of Salem polynomials and Salem numbers. It is easy to see that if $S$ is a Salem polynomial,
then ${\bf Q}[x]/(S)$ is a Salem field (see Proposition \ref{Salem implies Salem-like}), and this is the motivation for the terminology we chose. 

\medskip
Salem polynomials and Salem numbers play an important role in the dynamics of automorphisms of K3 surfaces. This study was initiated by  McMullen,
who gave the first examples of automorphisms of non-projective K3 surfaces with dynamical degree a Salem number, see \cite {Mc1}. The automorphisms constructed by McMullen induce Hodge endomorphisms, and his K3 surfaces have - in the terminology of the present paper - Salem multiplication.

\medskip We now come to the connection with number theory. 
In a 1975 paper, Stark obtained some striking results on Salem fields (not in this terminology); if $E$ is a Salem field, he constructed an explicit unit $\epsilon$ of
$E$  such that $E = {\QQ}(\epsilon)$, and he determined the index of the subgroup generated by $\epsilon$ and the units of the maximal totally real subfield of $E$ in the full group of units of $E$ (see \cite{Stark}, Theorem 2).  On the other hand,  Chinburg proved that Stark's unit $\epsilon$ is a Salem number (see \cite{Chinburg}). Combining these results, we see that Salem fields are generated by Salem numbers, and obtain the following result (see Proposition \ref{chinburg}):


\begin{prop}\label{chinburg'} Let $E$ be a Salem field. Then there exists a Salem polynomial $S$ such that $E = {\bf Q}[x]/(S)$.

\end{prop}

This proposition can be proved in a more direct way; 
we thank 
Chris Smyth for allowing us to include his proof,  more elementary than our original argument. Note however that Stark's unit (the ``Chinburg-Stark Salem number''
of the Salem field) is given explicity
in terms of arithmetic invariants of the Salem field, and this in turn will be useful in connection with complex dynamics.

\medskip
The fourth and last part of the paper (Section \ref{s:dyn}) 
is dedicated to the relations of Salem multiplication to complex dynamics.
Subsection \ref{bimer} is devoted to bimeromorphic automorphisms of elliptic HK manifolds (defined in Section \ref{s:HK}).
We start by recalling some results of McMullen and Oguiso, and continue by defining the notion of {\it Salem automorphism}. This is followed by a subsection dealing with the case of K3 surfaces, and the comparison of the results concerning Salem multiplication and Salem automorphisms. Indeed, a Salem automorphism of a K3 surface of algebraic dimension 0 always induces Salem multiplication, but the converse does not necessarily hold (see Example
\ref{ex:Smyth} and Example \ref{second smallest}). 
We also show that if an elliptic  HK manifold has real multiplication, then its group of bimeromorphic 
automorphisms is finite, and that the converse holds for K3 surfaces:

\begin{theo}
\label{thm:bir}
A K3 surface $X$ of algebraic dimension {\rm 0}  has finite ${\rm Bir}(X)=\Aut(X)$ if and only if $A_X$ is totally real. 
\end{theo}

\medskip This is proved in Corollary \ref{finite Bir}. 
The final result (Theorem \ref{thm:entropies}) combines complex dynamics and number theory: the set of
algebraic entropies of an elliptic HK manifold can be described via the arithmetic results of subsection \ref{Stark and Chinburg}.

\subsubsection*{Organization of the paper}

We review the necessary results about K3 surfaces,
Hodge structures of K3 type and their endomorphisms in Section \ref{s:K3}, and the notion of pseudo-polarized Hodge structure is introduced here.
Quadratic and hermitian forms and the concept of transfer are discussed in Section \ref{hermitian}. 
The study of the signatures of algebras with involution over the real numbers  in Section \ref{building blocks}
paves the way towards the proof of Theorem \ref{endos}
in Section \ref{proof}.
We prepare for the proof of Theorems \ref{RM} and \ref{SM} in Section \ref{s:period}
by setting up the bridge from suitable quadratic forms to pseudo-polarized K3 type Hodge structures
with RM and SM.

The closer look at  the interplay of signatures and transfer in Section \ref{signatures} 
allows for a characterization of the  quadratic forms of the desired shape
arising from totally real fields (Sections \ref{s:RM}, \ref{s:even}). This enables us to derive the geometric realizations on non-projective K3 surfaces at the end of  Section \ref{ss:RM}
to prove Theorem \ref{RM}.
The results for
real multiplication on non-projective hyperk\"ahler manifolds are proved 
in Section \ref{s:HK}. The results concerning Salem multiplication are stated in Theorem \ref{non-maximal theorem} and Theorem \ref{610}, and 
proved in \S \ref{prescribed Salem}, after some arithmetic results on hermitian forms over Salem fields. The final two sections are devoted to the role of Salem multiplication in complex dynamics (\S \ref{s:dyn}). Here we also prove Theorem \ref{thm:bir}.

\section{Hodge structures of  K3 type}
\label{s:K3}

\subsection{K3 surfaces} We refer to \cite{H} for the material in this section.

A K3 surface $X$ is a simply connected
2-dimensional compact complex manifold with trivial canonical bundle. By Siu's theorem \cite{Siu}, any
K3 surface is a K\"ahler manifold; 
its cohomology groups carry a Hodge structure,
that is, there are Hodge decompositions
$$
H^r(X,\CC)\,=\,H^r(X,\ZZ)\otimes_{\ZZ}\CC\,=\,
\oplus_{p+q=r} H^{p,q}(X),\qquad \overline{H^{p,q}(X)}\,=\,H^{q,p}(X)
$$
where the integers $p,q$ are non-negative. 
For $p=0,2$ one has $H^{2p}(X,\CC)=H^{p,p}(X)\cong\CC$ and $H^r(X,\ZZ)=0$ for $r=1,3$.
The interesting Hodge structure is the one on $H^2(X,\ZZ)$, which is a free $\ZZ$-module of rank $22$:
$$
H^2(X,\CC)\,=\,H^{2,0}(X)\,\oplus\,H^{1,1}(X)\,\oplus\,
H^{0,2}(X)~,\qquad \dim H^{2,0}(X)\,=\,H^{0,2}(X)\,=\,1~.
$$

The cup product $H^2(X,\ZZ)\times H^2(X,\ZZ)\rightarrow H^4(X,\ZZ)\cong\ZZ$ (the latter isomorphism is
natural since $X$ is oriented) is a perfect pairing.
This makes $H^2(X,\ZZ)$ into a unimodular lattice
which is even by Wu's formula, and has signature $(3,19)$.

\medskip Therefore one can characterize the lattice $H^2(X,\ZZ)$ as an abstract lattice, as follows.
Let us denote by $H$ the hyperbolic plane over $\ZZ$, and by $E_8$ the root lattice corresponding to the Dynkin diagram $E_8$; it is
well-known that $E_8$ is, up to isomorphism, the  unique unimodular even positive-definite lattice of rank $8$.
Set $$
\Lambda_\text{K3} =
H^{\oplus 3}\oplus E_8(-1)^{\oplus 2}.
$$ 

It follows from the above considerations that $H^2(X,\ZZ)$ and $\Lambda_\text{K3}$ are isomorphic as abstract lattices.

\medskip

The cup product on 
$H^2(X,\CC)=H^2_{DR}(X,\RR)\otimes_{\RR}{\CC}$ is given by the integral
$$
(\omega,\eta)\,=\,\int_X\omega\wedge \eta~.
$$
Since $H^{p,q}(X)\cong H^{p,q}_{\bar{\partial}}(X)$ in terms of the Dolbeault cohomology groups,
any element of $H^{p,q}(X)$ is represented by a 2-form of type $(p,q)$. Moreover, $H^{2,0}(X)=\CC\omega$
where $\omega$ is the class of a holomorphic 2-form, which is unique up to scalar multiple.
From this one finds that
$$
(\omega,\omega)\,=\,0,\qquad (\omega,\overline{\omega})\,>\,0,\qquad
H^{1,1}(X)\,=\,\left(H^{2,0}(X)\oplus H^{0,2}(X)\right)^\perp.
$$
In particular, the Hodge decomposition is completely determined by the class 
$\omega\in H^2(X,\CC)$.

\medskip
Set $\Lambda = \Lambda_\text{K3}$. The surjectivity of the period map for K3 surfaces asserts that, given any
$\omega\in \Lambda_\CC= \Lambda\otimes_\ZZ\CC$ 
satisfying
$(\omega,\omega)=0$ and $(\omega,\overline{\omega})>0$,
there exists a K3 surface $X$ with an isometry $\varphi: H^2(X,\ZZ)\rightarrow
\Lambda$ such that the complexification $\varphi_{\CC}$ maps  a holomorphic 2-form
on $X$ to $\omega$. The Torelli theorem asserts that $X$ is unique.

\medskip
The Picard group $\Pic(X)$ of $X$ is the group of line bundles, with tensor product, on $X$.
The Lefschetz (1,1)-Theorem asserts that
$$
\Pic(X)\,=\,H^{1,1}(X)\,\cap\,H^2(X,\ZZ)~.
$$
Equivalently, $\Pic(X)=\{c\in H^2(X,\ZZ):\,(\omega,c)=0\,\}$, since $\overline{(\omega,c)}=(\overline{\omega},\overline{c})$ and $\overline{c}=c$.
The K3 surface $X$ is projective if and only if there is a class $h\in \Pic(X)$ with $(h,h)>0$.

\subsection{The sublattices $\Pic(X)$ and $T_X$}\label{sublattices}
The transcendental lattice $T=T_X$ is the smallest primitive sublattice of $H^2(X,\ZZ)$ such that
$H^{2,0}(X)\subset T_\CC$. Here primitive means that $H^2(X,\ZZ)/T_X$ is torsion free.
One has (\cite[Lemma 3.3.1]{H}):
$$
T_X\,=\,\Pic(X)^\perp~.
$$
In case $\Pic(X)$ is a non-degenerate lattice, for example if $X$ is projective, one has an inclusion, with finite index,
$T_X\oplus \Pic(X)\subset H^2(X,\ZZ)$, but in general it can happen that $T_X\cap \Pic(X)\neq 0$
(see Lemma \ref{lem:ker}); in that case $T_X+ \Pic(X)$ does not have finite index in $H^2(X,\ZZ)$.
Notice that $T_X$ has a Hodge structure induced by the one on $H^2(X,\ZZ)$:
$$
T_{\CC}\,=\,\oplus_{p+q=2} T^{p,q},\qquad T^{p,q}\,:=\,T_\CC\cap H^{p,q}(X)~,
$$
one has $T^{2,0}=H^{2,0}(X)$, $T^{0,2} =H^{0,2}(X)$ and $T^{1,1}=(T^{2,0}\oplus T^{0,2})^\perp\subset T_\CC$.
We have the following possibilities for $T_X$ (cf.\ \cite[Proposition 4.11]{Kondo}).
Note that, to accommodate case (2),
we use the word `lattice' also for free $\ZZ$-modules with a degenerate bilinear form.

\begin{prop}\label{algdim}
\label{prop:Kondo}
Let $X$ be a K3 surface, let
$a(X)$ be the algebraic dimension of $X$, that is, the transcendence degree of the function field of $X$,
and let $r=rank(T_X)$.
Then exactly one of the following holds.
\begin{enumerate}
 \item The lattice $T_X$ is non-degenerate with signature $(2,r-2)$
 and $X$ is projective,
 so $a(X)=2$.
 \item The lattice $T_X$ has a kernel $K$ of rank one, the quotient $T_X/K$ has signature $(2,r-3)$,
 and $a(X)=1$.
 \item The lattice $T_X$ has signature $(3,r-3)$
and  $a(X)=0$.
\end{enumerate}
\end{prop}

Note that in case (2), $T_X$ is not simple as a Hodge structure,
since it has the non-trivial (primitive) sub-Hodge structure given by $K$
(which is pure of type $(1,1)$)
-- without a complementary Hodge structure.

A polarized, rational, Hodge structure is a direct sum of simple rational Hodge structures.
But, if $a(X)=1$, the Hodge structure on $T_X$ is not polarized by the intersection form
and the existence of $K$ shows that it does not admit any polarization.
In \S \ref{hodge 1} we will show that $K$ is the only non-trivial primitive sub-Hodge structure of $T_X$.

\subsection{The Hodge endomorphisms} \label{hodge endo}

Let $T=T_X$ be the transcendental lattice of a K3 surface and let $T_\QQ:=T\otimes_\ZZ\QQ$.
We are interested in the $\QQ$-linear maps $f:T_\QQ\rightarrow T_\QQ$ which preserve
the Hodge structure on $T$, that is, the $\CC$-linear extension of $f$ satisfies
$f(T^{p,q})\subset T^{p,q}$.
The algebra of such Hodge endomorphisms is denoted by $A_T$:
$$
A_T\,:=\,\{f\,\in\,\End_\QQ(T_\QQ):\;f(T^{p,q})\subset T^{p,q}\,\}~.
$$
If we do not want to make reference to $T$, we shall also write $A_X$.
Notice that $f(T^{2,0})\subset T^{2,0}$ so there is a homomorphism of algebras
$$
\delta:\,A_T\,\longrightarrow \End_\CC(T^{2,0})=\CC,\qquad
f\,\longmapsto \, \delta(f)\,:=\,f_{|T^{2,0}}~.
$$
Note that  \cite[Lemma 3.3.3]{H} asserts that this map is injective
(regardless of $T$ being simple -- this only needs that there is no proper sub-Hodge structure containing a non-trivial $(2,0)$-part, which is the case by minimality of $T$).
As $\End_\CC(T^{2,0})\cong\CC$,
this implies that $A_T$ is a (commutative) field, thus $T_{\QQ}$ is an $A_T$-vector space.

\medskip
In the case where $a(X)=2$, so $X$ is projective, the algebra $A_T$ was studied extensively in  \cite{Z} and \cite{BGS}.

\subsection{K3 surfaces of algebraic dimension 1}\label{hodge 1}

In this case the kernel $K$ of the lattice $T = T_X$ has rank one by Proposition \ref{prop:Kondo}.
This sublattice can be characterized as follows.

\begin{lemma}
\label{lem:ker}
The kernel $K$ of the transcendental lattice $T$ is also
$$
K\,=\,T\cap T^{1,1}\,=\,T\,\cap\,\Pic(X)~.
$$
\end{lemma}

\noindent
{\bf Proof.}
The intersection form is non-degenerate on $T^{2,0}\oplus T^{0,2}$
and thus $K\subset T^{1,1}$, that is, $K\subset T\cap T^{1,1}$.
Next we observe that $T\cap T^{1,1}\subset T\cap\Pic(X)$.
Finally, using $T=\Pic(X)^\perp$, we have $T\cap \Pic(X)=T\cap T^\perp$, and $T\cap T^\perp=K$,
the kernel of $T$.
So we proved $K\subset T\cap T^{1,1} \subset T\cap \Pic(X)= K$ and thus all inclusions are equalities.
\qed

\

Since any $f\in A_T$ maps both $T$ and $T^{1,1}$ into themselves, we see that $f(K)\subset K$. Therefore we
have a restriction homomorphism
$$
A_T\,\longrightarrow\, \End_\QQ(K)\,\cong\,\QQ~.
$$
Moreover, $f$ induces a Hodge endomorphism on $T/K$, which has the Hodge structure induced by
$T$, so combined with the homomorphism above we have a homomorphism
\begin{eqnarray*}
\varphi:
A_T & \longrightarrow\ &A_{T/K}\,\times\,\QQ~,\\
 f \;\; & \mapsto &\;\;\;\,  ([f], f|_K)~.
\end{eqnarray*}

\begin{remark}
\label{rem:a=1}
The quotient $T/K$ has a non-degenerate bilinear form induced by the (restriction to $T$) of
the intersection form, and this form is a polarization on $T/K$. Thus we can apply all the results
of \cite{Z} and \cite{BGS}
on endomorphisms of polarized K3-type Hodge structures to $A_{T/K}$, so $A_{T/K}$ is either a RM or a CM field.
\end{remark}

We now show that any sub-Hodge structure of $T_\QQ$ is either trivial or it is $K_\QQ$.
Moreover, $A_T$ consists of scalar multiples of the identity.

\begin{prop}\label{prop:hodge 1}
Let $T'\subset T_\QQ$ be a sub-Hodge structure, then  $T'$ is one of $0$, $K_\QQ$ and $T_\QQ$.

The algebra of Hodge endomorphisms of $T_\QQ$ is $A_T\cong\QQ$.
\end{prop}

\begin{proof}
If $T^{2,0}\subset T'_\CC$ then $T'=T_\QQ$ by minimality of $T$.
Otherwise, $(T')^{2,0}=0$ and then $T'\subset T^{1,1}$. As $T^{1,1}$ is orthogonal to $T^{2,0}$, 
so is the $\QQ$-vector space $T'$ and therefore $T^{2,0}\subset (T')^\perp_\CC$.
By minimality of $T$ we get $(T')^\perp=T_\QQ$. Thus $T'\subset Ker(T_\QQ)=K_\QQ$. Since $K_\QQ$ is 1-dimensional, $T'=0$ or $T'=K_\QQ$.

Let $f\in A_T$, then $f(K_\QQ)\subset K_\QQ$ and thus $f|_{K_\QQ}=c_f \id_{K_\QQ}$ for some $c_f\in\QQ$.
The Hodge endomorphism
$f-c_f \id_T$ is then zero on $K_\QQ$, so its image, a sub-Hodge structure of $T_\QQ$,
cannot be $T_\QQ$, hence it is contained in $K_\QQ$.
But $(f-c_f \id_T)(T^{2,0})\subset T^{2,0}$ whereas $K_\QQ\cap T^{2,0}=0$, so $(f-c_f \id_T)(T^{2,0})=0$. By minimality of $T$,
the $\QQ$-subspace $\ker(f-c_f \id_T)$ must be all of $T_\QQ$ and therefore $f=c_f \id_T$.
\qed
\end{proof}

\subsection{Pseudo-polarized Hodge structures}\label{hodge pp}

For the remainder of this paper, we focus on K3 surfaces $X$ of algebraic dimension $a(X)=0$.
In this case the intersection form on $H^2(X,\ZZ)$ defines a non-degenerate bilinear form of signature
$(3,r-3)$ on $T$. Notice that it is not a polarization of the Hodge structure $T$;
hence we have to work out new proofs of all results analogous to the polarized case.

\medskip

For later use, we slightly generalize these Hodge structures, dropping the condition that they are
sub-Hodge structures of $H^2(X,\ZZ)$ for a K3 surface $X$, so that we can also use the results
we obtain for the more general class of hyperk\"ahler manifolds in Section \ref{s:HK}.
In particular, we do not impose any condition on the rank of $T$.

\begin{defn}\label{def:pseudo}
\label{def:T}
A $\QQ$-vector space $T_\QQ$ of dimension $r$ with a non-degenerate bilinear form $(.,.)$ of signature $(3,r-3)$
is called a pseudo-polarized rational Hodge structure of K3 type if $T_\QQ$
has a Hodge decomposition
$$
T_\CC\,=\,T^{2,0}\,\oplus\,T^{1,1}\,\oplus\,T^{0,2},\qquad\mbox{with}\quad
\overline{T^{p,q}}\,=\,T^{q,p}~,
$$
and $\dim T^{2,0}=1$, and such that  the subspace $T^{1,1}$ is perpendicular to both
$T^{2,0}$ and $T^{0,2}$ whereas if $\omega\in T^{2,0}$ is a basis of $T^{2,0}$ then
\begin{eqnarray}
\label{eq:()=0}
(\omega,\omega)\,=\,0,\qquad (\omega,\overline{\omega})\,>\,0
\end{eqnarray}
(notice that then also $(\overline{\omega},\overline{\omega})=0$).
The Hodge structure $T_\QQ$ is said to be {\it simple} if $\omega^\perp\cap T_\QQ\,=\,\{0\}$.

\end{defn}

Analogously to the K3 surface case, we define the algebra of {\it Hodge endomorphisms} of $T$ by setting 

$$
A_T\,:=\,\{f\,\in\,\End_\QQ(T_\QQ):\;f(T^{p,q})\subset T^{p,q}\,\}~.
$$

\medskip
There is an `adjoint' involution on $\End_{\QQ}(T_\QQ)$ defined by
\begin{eqnarray}
\label{eq:adjoint}
f\,\longmapsto f',\qquad\mbox{where}\quad (fx,y)\,=\,(x,f'y)~,
\end{eqnarray}
for $f\in \End_{\QQ}(T_\QQ)$ and all $x,y\in T_\QQ$.
We now show that this involution preserves the subalgebra $A_T$ of $\End_{\QQ}(T_\QQ)$.

\begin{prop} \label{adjoint}
Let $T_\QQ$ be a simple pseudo-polarized Hodge structure of K3 type.
(For example, $T_\QQ=T_{X,\QQ}$ where $a(X)=0$).
\begin{enumerate}
\item
Let $T'\subset T_\QQ$ be a sub-Hodge structure, then $T'=0$ or $T_\QQ$.
\smallskip
\item
The commutative $\QQ$-algebra $A_T$ is a number field.
\smallskip
\item For $f\in A_T$ also $f'\in A_T$.
\smallskip
\item Let $\delta:A_T\hookrightarrow\CC$ be the action of $A_T$ on
$T^{2,0}$. Then for all $f\in A_T$ we have
$$
\delta(f')\,=\,\overline{\delta(f)}~.
$$
\end{enumerate}
\end{prop}

\noindent
{\bf Proof.}
(1)
Let $T'\subset T_\QQ$ be a sub-Hodge structure. If $(T')^{2,0}=0$, 
then $T'\subset T^{1,1}=(T^{2,0}\oplus T^{0,2})^\perp$.
Therefore $(t',\omega)=0$ for all $t'\in T'$ so that $T'=0$ by simplicity of $T_{\QQ}$.
Otherwise $T^{2,0}\subset T'_\CC$, so that $((T')^\perp)^{2,0}=0$ and we just showed that then the
sub-Hodge structure $(T')^\perp=0$ so $T'=T_\QQ$.


(2)
In \S \ref{hodge endo} we already observed that $A_T$ is a (commutative) field, since it
is also a subalgebra of the finite dimensional $\End(T_\QQ)$, it is a number field.

(3)
Recall that for $f\in A_T$ we have $f(\omega)=\delta(f)\omega$.
Then
$$
\delta(f)(\omega,y)\,=\,(f(\omega),y)\,=\,(\omega, f'(y))~.
$$
Hence $(\omega,y)=0$ implies $(\omega,f'(y))=0$ so that $f'(\omega^\perp)=\omega^\perp$ and similarly
 $f'(\overline{\omega}^\perp)=\overline{\omega}^\perp$.
Since $T^{1,1}=\omega^\perp\cap\overline{\omega}^\perp$ it follows that $f'(T^{1,1})=T^{1,1}$.
Then also $f'((T^{1,1})^\perp)=(T^{1,1})^\perp$, this is a two dimensional complex subspace of $T_{\bf C}$
with a non-degenerate quadratic form. The subspaces $T^{2,0}$, $T^{0,2}$ are isotropic subspaces and thus
they are either mapped into themselves or they are interchanged. Taking $y=\omega$ above, we see that
$0=(\omega,\omega)=(\omega,f'(\omega))$ hence $f'(\omega)$ lies in the isotropic subspace $T^{2,0}$
and thus $f'(T^{2,0})=T^{2,0}$. Therefore $f'\in A_T$.

(4)
This follows from $(\omega,\overline{\omega})\neq 0$ and
$$
\delta(f)(\omega,\overline{\omega})\,=\,(f(\omega),\overline{\omega})\,=\,
(\omega,f'(\overline{\omega}))\,=\,\overline{\delta(f')}(\omega,\overline{\omega})~.
$$
\qed



\medskip

Setting $E=A_T$, we  follow standard number field notation to denote the $\QQ$-linear involution obtained from Proposition \ref{adjoint} (3) by $E\ni e\mapsto \bar e$.
Then \eqref{eq:adjoint} gives the adjoint condition
\begin{eqnarray}
\label{eq:adj}
(ex,y)=(x,\overline{e}y)\;\;\; \forall\, x,y\in T_\QQ, \;\; e\in E.
\end{eqnarray}

\medskip
In the sequel, we will use the terminology {\it pseudo-polarized Hodge structures}  for pseudo-polarized Hodge structures of K3 type.

\section{Quadratic and hermitian forms, transfer}\label{hermitian}

If $X$ is a pseudo-polarized (or projective) K3 surface,  the intersection form of $X$ induces a non-degenerate
bilinear form on $T_{\QQ}=T_{X,\QQ}$. 
As we will see, the adjoint condition \eqref{eq:adj} allows us to identify the quadratic space
$(T_\QQ,(\cdot,\cdot))$
with the transfer $\TT(W)$ of a quadratic or hermitian space $W$ over $E$.

\medskip
We start by recalling some notions concerning quadratic forms and algebras with involution. Let $k$ be a field of characteristic $\not = 2$.
A {\it quadratic form} over $k$ is by definition a pair $V = (V,q)$, where $V$ is a finite
dimensional $k$-vector space and
$q: V \times V \to k$ is a non-degenerate symmetric bilinear form.

\medskip
If $A$ is a commutative $k$-algebra, a $k$-linear {\it involution} is by definition a $k$-linear map $\iota : A \to A$ such
that $\iota(ab) = \iota(a)\iota(b)$ for all $a,b \in A$, and such that $\iota \circ \iota$ is the identity. An {\it \'etale algebra} is by definition a product of
a finite number of separable field extensions of $k$ each of which has finite degree;
by definition, it is equipped with a trace map ${\rm Tr}: A\to k$.
If $B$ is an \'etale $k$-algebra and we set $A = B \times B$, with involution $\iota$ exchanging the two
copies of $B$, then the algebra with involution $A$ is said to be {\it split}. 

\medskip
The trace form
\[
q: A\times A \to k,\;\;\; 
q(a,a') = {\rm Tr} (aa')
\]
defines a quadratic form on $A$ which gives rise to the 
discriminant $\Delta_A$ of $A$, defined as an element
of $k^{\times}/k^{\times 2}$.
Note that a split algebra $A$ has $\Delta_A=1$ since $\Delta_A=\Delta_B^2$.

%
%
%

\medskip
Let now $E$ be an algebraic number field of degree $d$, and let $e \mapsto \overline e$ be a $\bf Q$-linear involution, possibly the identity. Let $E_0$
be the fixed field of the involution; $E_0 = E$ if the involution is the identity, otherwise $E/E_0$ is a quadratic extension. A {\it hermitian
form} is a pair $(W,h)$, where $W$ is a finite dimensional $E$-vector space, and $h : W \times W \to E$ is a sesquilinear form
so that $h$ is $E$-linear in the first variable and $\overline {h(x,y)} = h(y,x)$
for all $x,y \in W$; if the involution is the identity, then $(W,h)$ is a quadratic form over $E$.

\medskip
Every hermitian form over $E$ can be diagonalized, i.e.\ in a suitable basis $(W,h)$ is represented by 
a diagonal matrix diag$(\alpha_1,\hdots,\alpha_n)$
 for some $\alpha_i \in E_0^{\times}$.
 In short, we write $(W,h)\cong  \langle \alpha_1,
\dots,\alpha_n\rangle$.
The {\it determinant} of $W = (W,h)$ is by definition the product
$\alpha_1 \cdots \alpha_n$ considered as an element of $E_0^{\times}/{\rm N}_{E/E_0}(E^{\times})$
if the involution is non-trivial,
and of $E^{\times}/E^{\times 2}$ if it is the identity.

\medskip
In the sequel, we have to deal with algebras with involution obtained from $E$ by base change to $\bf R$. Set $E_{\bf R} = E \otimes_{\bf Q}{\bf R}$, and
note that the involution extends ${\bf R}$-linearly to this \'etale algebra.

\begin{lemma}\label{transfer lemma} Let $k = {\bf Q}$ or ${\bf R}$, and let $A = E$ or $A = E_{\bf R}$, as above. Let $(U,q)$ be a quadratic form over $k$,
and suppose that $U$ has a structure of $A$-module. The following are equivalent

\medskip
{\rm (i)} For all $x,y \in U$ and all $\alpha \in A$, we  have
$$
q(\alpha x,y)\, =\, q(x,\overline {\alpha} y).
$$

\medskip
{\rm (ii)} There exists a hermitian form $h : U \times U \to A$ such that for all $x,y \in U$, we have
$$
q(x,y)\, =\, {\rm Tr}_{A/k}(h(x,y)).
$$

\end{lemma}

\noindent
{\bf Proof.} This is proved in  \cite{BGS}, Lemma 5.1, for $k = \bf Q$; the proof is the same when $k = \bf R$.
\qed

\begin{defn}
Let $W$ be a finite dimensional $E$-vector space, and let $Q : W \times W \to E$ be a quadratic form if $E = E_0$, and
a hermitian form with respect to the involution $e \mapsto \overline e$ if $E \not = E_0$.
We denote by ${\rm T}_E(W)=(W,q)$ the quadratic form over ${\bf Q}$ defined by
$$
q:\,W \times W\, \longrightarrow {\bf Q},\ \ \  q(x,y) ={\rm Tr}_{E/{\bf Q}}(Q(x,y)),
$$
called the {\it transfer} of $W$, or more precisely of $(W,Q)$.
\end{defn}

\medskip
If $E$ is fixed, we write ${\rm T}(W)$ instead of ${\rm T}_E(W)$.

\begin{example}\label{transfer example}  Let  $d \geqslant 2$ be a square-free integer such that
$d=u^2+v^2$ with $u,v\in\ZZ$, and let $E=\QQ(\sqrt d)$ be the corresponding real quadratic field, with the trivial involution. 
Then $U=\langle 1,1 \rangle$ admits an action by $E$ satisfying the adjoint condition \eqref{eq:adj}.
Indeed,  
consider the linear action $\varphi$
given by the matrix $A=\begin{pmatrix}
u & v\\
v & -u
\end{pmatrix}$.
Then $\varphi^2=d \cdot \id$, and since $A$ is symmetric, the adjoint condition  \eqref{eq:adj} is satisfied.

\medskip

As predicted by Lemma \ref{transfer lemma}, there is a one-dimensional quadratic form $W$ over $E$ such that
$U\cong\TT(W)$; indeed, under the identification of $\QQ$-bases 
$(1,0) \mapsto 1, \; (u,v) = \varphi(1,0)\mapsto\sqrt d$,
 we can take $\alpha = (d+u\sqrt d)/(2d)$ and set $W = \langle \alpha \rangle$.
\qed

\end{example}

Note that by Lemma \ref{transfer lemma} we have 

\begin{prop} Let $(U,q)$ be a quadratic form over $\QQ$. Suppose that $U$ has a structure of $E$-module
and that $q(\alpha x,y)\, =\, q(x,\overline {\alpha} y)$ for all $x,y \in U$ and all $\alpha \in E$. 
Then there exists a quadratic or hermitian form $W$ over $E$ such that $U \simeq T(W)$.

\end{prop}

We conclude this section by a lemma that is used in several sections. Let $\Delta_E$ be the discriminant of $E$. If $E \not = E_0$, set $d_0 = [E_0:{\bf Q}]$ (hence $d_0 = {\frac d2}$).
The following is proved in Lemma 6.1 of \cite{BGS}.

\
\begin{lemma}\label{invariants bis}
{\rm (i)} ${\rm dim}_{{\bf Q}}({\rm T}(W)) = d\cdot {\rm dim}_E(W)$.

\medskip
{\rm (ii)} If $E = E_0$, then  ${\rm det}({\rm T}(W)) = \Delta_E^{{\rm dim}_E(W)}{\mathrm N}_{E/{\bf Q}}({\rm det}(W))$ in ${\bf Q}^{\times}/{\bf Q}^{\times 2}$.

\medskip
{\rm (iii)} If $E \not = E_0$, then  ${\rm det}({\rm T}(W)) = [(-1)^{d_0} \Delta_E]^{{\rm dim}_E(W)}$ in ${\bf Q}^{\times}/{\bf Q}^{\times 2}$.

\end{lemma}

\begin{example} 
\label{ex:(1,1)}
Let  $d \geqslant 2$ be a square-free integer, and let $E=\QQ(\sqrt d)$ be the corresponding real quadratic field, with the trivial involution.
As in Example \ref{transfer example}, set $U=\langle 1,1 \rangle$. As a converse of Example \ref{transfer example}, let us show that if
there exists a one-dimensional quadratic form $W$ over $E$ such that
$U\cong\TT(W)$, then there 
exist integers $u,v$ such that
$d=u^2+v^2$. 

Indeed, writing $W=\langle \alpha\rangle$ for $\alpha=a+b\sqrt d\in E$, we
obtain from Lemma \ref{invariants bis} (ii) (or by direct computation) that,
\[
1 = \det(U)  = \Delta_E \det(W) = d\, {\rm N}_{E/\QQ}(\alpha) = d (a^2-db^2) \;\;\; \text{ in } \; \QQ^*/{\QQ^*}^2.
\]
Division by $a^2$ yields, since obviously $ab\neq 0$,
\[
d = u^2 + v^2 \;\;\; \text{ for some } \; u, v\in{\QQ^*}^2.
\]
This implies that there exist integers $x,y$ such that $d = x^2 + y^2$. Indeed, the previous equality implies that there exists an integer $m$ such that
$m^2d = t^2 + z^2$ with $t,z \in {\ZZ}$. 
By Fermat's theorem on sums of two squares of integers, this implies that there does not exist any
prime number $p\equiv 3\mod 4$ that divides $m^2d$ with an odd exponent. 
One infers that $d$ does not have any such prime factor either, hence, by the converse implication of Fermat's theorem, 
$d$ is a sum of two squares of integers as stated.

\end{example}


\section{Quadratic forms}

In the next sections, we need some additional notions and facts concerning quadratic forms. Let $k$ be a field of characteristic $\not = 2$.
If  $V = (V,q)$ is a quadratic form over $k$, then $V$
can be diagonalized: there exist $a_1,\dots,a_n \in k^{\times}$ such that $V$ is isomorphic to the diagonal quadratic
form $\langle a_1,\hdots,a_n \rangle$. The {\it determinant} of $V$ is by definition ${\rm det(}V) = \underset{i} \prod a_i$ in $k^{\times}/k^{\times 2}$.
Let ${\rm Br}(k)$ be the Brauer group of $k$, and let ${\rm Br}_2(k)$ be the subgroup of ${\rm Br}(k)$ consisting of the elements of order $\leqslant 2$. 
The {\it Hasse invariant} of $V$ is $w(V) = \underset{i < j} \sum (a_i,a_j)$ in ${\rm Br}_2(k)$, where $(a_i,a_j)$ is the class
of the quaternion algebra determined by $a_i$ and $a_j$. We use the additive notation for the
abelian group ${\rm Br}_2(k)$. If $V'$ is another quadratic form, then
$$w(V \oplus V') = w(V) +
w(V') + ({\rm det}(V),{\rm det}(V')).
$$

\medskip We refer to \cite[\S 1]{BGS} for a concise review of the results on quadratic forms and Brauer groups needed in the present paper.

\medskip
The following lemma is well-known.

\begin{lemma}\label{4 copies} Let $U_0$ be a quadratic form over $\QQ$, and set $$U = U_0 \oplus U_0 \oplus U_0 \oplus U_0.$$ Then
$w(U) = 0$ in ${\rm Br}_2(\QQ)$.

\end{lemma}

\noindent
{\bf Proof.} Set $V = U_0 \oplus U_0$, and $d = {\rm det}(U_0))$. We have ${\rm det}(V) = {\rm det}(U_0)^2 = 1$ in ${\bf Q}^{\times}/{\bf Q}^{\times 2}$ and
$w(V) = w(U_0) + w(U_0) + (d,d) = (d,d)$. Note that
$U = V \oplus V$, hence $w(U) = w(V) + w(V) = 0$.
\qed

\medskip

\begin{example}\label{quadratic example} Let $d > 0$ be an integer, and set $E = \QQ(\sqrt d)$ with trivial involution. 
We denote by $I_n$ the $n$-dimensional unit form $\langle 1,\dots,1 \rangle$.

\medskip
(1) Set $W = \langle \sqrt d \rangle$, and set $U = {\rm T}(W)$. Then $U \simeq H$. Indeed, by Lemma \ref{invariants bis} we have
${\rm dim}(U) = 2$ and ${\rm det}(U) = -1$ in ${\bf Q}^{\times}/{\bf Q}^{\times 2}$. This implies that $U \simeq H$.

\medskip
(2) Let $W_0 = \langle 1 \rangle$ and $W = I_4$, as quadratic forms over $E$; set $U_0 = {\rm T}(W_0)$ and $U = {\rm T}(W)$.
We have $U \simeq I_8$. Indeed, ${\rm dim}(U) = 8$, and ${\rm det}(U) = {\rm det}(U_0)^4 = 1$  in ${\bf Q}^{\times}/{\bf Q}^{\times 2}$.
The quadratic form $U$ is positive definite, and by Lemma \ref{4 copies} we have $w(U) = 0$ in ${\rm Br}_2(\QQ)$;
therefore $U \simeq I_8$. \qed

\end{example}

Even though imaginary quadratic fields do not play a role in this paper, we record the following example, to point out the contrast with
the real quadratic case.

\begin{example} Let $d < 0$ be an integer, and set $E = \QQ(\sqrt d)$ with involution induced by complex conjugation,
i.e.\ sending $\sqrt d \mapsto -\sqrt d$; let $W$ be a one-dimensional hermitian form over $E$, and set
$U = {\rm T}(W)$.
Then
$W=\langle a \rangle$ with $a \in \QQ$, and $U \simeq \langle 2a,2a \rangle$. In particular, $U$ is positive definite if $a >0$, and negative
definite if $a <0$, and the hyperbolic plane $H$ cannot be realized as $ {\rm T}(W)$ for any hermitian form $W$ over $E$.

\end{example}

We record an
example containing special cases that will be used in the sequel.

\begin{example}\label{HW} {\rm (i)} Let $H$ be the hyperbolic plane over $\QQ$, and let $n \geqslant 1$ be an integer. We have
$w(H^n) = 0$ if $n \equiv \ 0, 1  \ {\rm (mod \ 4)}$, and $w(H^n) = (-1,-1)$ if $n \equiv \ 2, 3  \ {\rm (mod \ 4)}$.

\medskip (ii) Set $\Lambda_{3,19} = H_0^3 \oplus E_8^2$,
where $H_0$ is the hyperbolic plane over $\ZZ$, and $E_8$ is the negative $E_8$-lattice. Note that if
$X$ is a complex K3 surface, then $H^2(X,\ZZ)$, with its
intersection form, is a lattice isomorphic to $\Lambda_{3,19}$.
Set $V_{K3} = \Lambda_{3,19} \otimes_{\ZZ} \QQ$. We have $w(V_{K3}) = (-1,-1)$.

\medskip Indeed, 
$V_{K3} \simeq (-I_8)^2 \oplus H^3$.
Therefore we have
$$w(V_{K3}) =   w((-I_8)^2 ) +  ({\rm det}((-I_8)^2 ),{\rm det}(H^3)) + w(H^3),$$and $w((-I_8)^2 ) = 0$, hence $w(V_{K3}) = (1,-1) + (-1,-1) = (-1,-1)$.

\end{example}

\section{Salem fields and Salem numbers}
\label{s:slf}

Recall the notion of Salem field from the  introduction (Definition \ref{def:Salem'});  an algebraic number field $E$ is a Salem field if
$E$ has two real embeddings, the other embeddings being complex, the degree of $E$ is $\geqslant 4$, and $E$ has a $\QQ$-linear involution with totally real fixed field.

\subsection{The discriminant of a Salem field}

The aim of this section is to prove some lemmas on discriminants of Salem fields that will be needed in the next sections; we start with a characterization of these fields.

\begin{lemma}
\label{lem:Slf}

Let $E$ be an algebraic number field. The following are equivalent

\medskip
{\rm (i)} $E$ is a Salem field.

\medskip
{\rm (ii)} There exists a totally real number field $E_0$ of degree $\geqslant 2$ and an element $\alpha \in E_0$
that is positive at exactly one embedding of $E_0$ such that $E = E_0 (\sqrt \alpha)$.

\end{lemma}

\noindent
{\bf Proof.} Let us prove that (i) implies (ii). Since $E$ is a Salem field, it has a non-trivial $\bf Q$-linear involution; let $E_0$ be the
fixed field of this involution. By hypothesis, $E_0$ is totally real. The extension $E/E_0$ is of degree 2, hence there exists
$\alpha \in E_0$ such that $E = E_0 (\sqrt \alpha)$. The field $E$ has exactly two real embeddings, the other embeddings are complex. This
implies that exactly one of the real embeddings of $E_0$ extends to (a pair of) real embeddings of $E$
while the others extend to (pairs of) complex embeddings; hence
$\alpha$ is positive at exactly one embedding of $E_0$, the one that extends to the pair of real embeddings, and it is negative at all the others.
This implies (ii).

\medskip
We now show that (ii) implies (i). Since the degree of $E_0$ is at least two and $E$ is a quadratic extension of $E_0$, the degree of $E$ is at least 4.
Let $e \mapsto \overline e$
be the involution of $E$ sending $\sqrt \alpha$ to $-\sqrt \alpha$. This is a non-trivial $\bf Q$-linear involution with fixed field $E_0$. By
hypothesis, $\alpha \in E_0$
 is positive at exactly one embedding and negative at all the others, hence $E$ has
two real embeddings, and all its other embeddings are complex. This implies that $E$ is a Salem field, hence (i) holds.
\qed

\begin{lemma}\label{negative} Let $E$ be a Salem field. Then $\Delta_E$ is negative if and only if the degree of $E$ is divisible by $4$.

\end{lemma}

\noindent
{\bf Proof.} Let us write $E = E_0 (\sqrt \alpha)$ for some $\alpha \in E_0$ that is positive at exactly one embedding of $E_0$. We have
$\Delta_E = {\rm N}_{E_0/{\bf Q}}(\alpha)\Delta_{E_0}^2$. 
On the
other hand, $\alpha$ is positive at one real place and negative at all the others. If the degree of $E_0$ is even, then  ${\rm N}_{E_0/{\bf Q}}(\alpha)$ is negative,
hence $\Delta_E < 0$; if the degree of $E_0$ is odd, then ${\rm N}_{E_0/{\bf Q}}(\alpha)$ is positive,
hence $\Delta_E >0$.
\qed

\medskip

Recall from the introduction (see Definition \ref{definition Salem polynomial}) that a Salem polynomial is a monic irreducible polynomial $S \in {\bf Z}[x]$ of degree $\geqslant 4$ such that
$S(x) = x^{{\rm deg}(S)} S(x^{-1})$ and that $S$ has exactly two roots not on the unit circle, both positive
real numbers.

\begin{lemma}\label{Salem discriminant} Let $S$ be a Salem polynomial of degree $2d$, and let $E = {\bf Q}[x]/(S)$.
Then $\Delta_E = (-1)^{d}S(1)S(-1)$
in ${\bf Q}^{\times}/{\bf Q}^{\times 2}$.
\end{lemma}

\noindent
{\bf Proof} Let $q : E \times E \to \QQ$ be the quadratic form given by $q(x,y) = {\rm Tr}_{E/{\QQ}}(x \overline y)$, where
$e \mapsto \overline e$ is the involution
induced by $X \mapsto X^{-1}$; then ${\rm det}(q) = S(1)S(-1)$ (see for instance \cite{GM}, Proposition A3). Note that, in the notation of \S \ref{hermitian}, we have
$q = {\rm T}(W)$, where $W$ is the one-dimensional hermitian form $\langle 1 \rangle$ over $E$. By Theorem  \ref{realization} (i) we have
${\rm det}(q) = (-1)^{d}\Delta_E$, hence $\Delta_E = (-1)^{d}S(1)S(-1)$ modulo squares.
\qed

\medskip

The reason of the terminology ``Salem field'' is that Salem numbers generate fields with this property. 
Recall that the unique real root $\lambda> 1$ of a Salem polynomial is called a {\it Salem number}.


%
%


\begin{prop}\label{chinburg} \label{Salem implies Salem-like} 
Let $E$ be a number field.
Then $E$ is a  Salem field if and only if there exists a Salem polynomial $S$ such that $E = {\bf Q}[x]/(S)$.
\end{prop}

\noindent
{\bf Proof.}  
Let $S$ be a Salem polynomial and $E = {\bf Q}[x]/(S)$.
Let $e \mapsto \overline e$ be the involution
induced by $X \mapsto X^{-1}$. Let us write $E = {\bf Q}(\lambda)$, with $\overline {\lambda} = \lambda^{-1}$. Set $\alpha = (\lambda - \lambda^{-1})^2$ and $E_0=\fQ[x]/(f)$ where $f\in\fQ[x]$ is determined by the property
that $f(x+x^{-1}) = x^{-\deg(S)/2}  S(x)$.
By construction, we have $E = E_0(\sqrt \alpha)$ and it is easy to check that $\alpha$ is positive at one real embedding of $E_0$ and negative at the others.
Therefore $E$ is a Salem field (see Lemma \ref{lem:Slf}).

\medskip

For the converse direction, 
we thank Chris Smyth for allowing us to include his proof,  more elementary than our original argument.

Let $E$ be a Salem field and set $2d = [E:{\QQ}]$. As usual, we denote by $E_0$ the fixed field of the involution of $E$. Let $\alpha \in E_0$ be 
positive at exactly one embedding of $E_0$ such that $E = E_0 (\sqrt \alpha)$ (see Lemma \ref{lem:Slf}). The field $E_0$ is totally real of degree $d$,
with conjugate fields $E^1 = E_0$, $E^2, \hdots, E^d$. For any $\gamma = \gamma_1 \in E_0$, we let $\gamma_i$ denote its conjugate in $E^i$, for
$i = 1,\dots,d$. By Dirichlet's unit theorem, the group of units of $E_0$ has rank $d-1$, say with generating set $U$. The field $E$ has $r = 2$ real
places and $s = d-1$ complex places, so the rank of the group of units of $E$ is $r+s-1 = d$. Now let $u$ be any unit of $E$ such that the
unit group generated by $U \cup \{u\}$ has rank $d$. Then $u \not \in E_0$, and is not of the form $ b\sqrt \alpha$ for any $b \in E_0$.
Thus on writing $u = a + b\sqrt{\alpha}$ for $a,b \in E_0$, where $a \not = 0$, we see from $u^2 = a^2 + \alpha b^2 + 2ab \sqrt{\alpha}$ that also $u^2 \not \in E_0$, and
$U \cup \{u^2\}$ has rank $d$. Now the $2d$ conjugates of $u$ are, say, $u_i = a_i + b_i \sqrt {\alpha}$ and $u_{d+i} = a_i - b_i \sqrt {\alpha}$
for $i = 1,\dots, d$. Hence, by swapping $u_1$ and $u_{d+1}$ and (or) replacing $u$ by its square where necessary, we can assume that
$u_1 > u_{d+1} > 0$.

Now set $\beta = u_1/u_{d+1}$; then $\beta \in E$, 
it  is a unit and it is greater than 1. Its conjugates are $u_1/u_{d+1}$, its reciprocal $u_{d+1}/u_1$,
as well as for $i=2,\dots,d$ the conjugates $u_i/u_{d+i}$ and $u_{d+i}/u_{i}$. Moreover  $u_i/u_{d+i}$ and $u_{d+i}/u_{i}$
are complex conjugates of each other for $i=2,\dots,d$, hence  of modulus 1.
Since $\beta$ has only one conjugate greater than 1 (that is, $\beta$ itself), its degree must be $2d$. 
Let now $S$ be the minimal polynomial of $\beta$; we have  $E = {\bf Q}[x]/(S)$, as claimed.
\qed

\bigskip

We record some further results concerning Salem fields and Salem numbers.

\begin{prop} 
\label{prop:E=F}
Let $E$ be a Salem field, and let $F$ be a subfield of $E$. If $F$ is a Salem field, then $F = E$.

\end{prop} 

\noindent
{\bf Proof.} Since $E$ is a Salem field, Proposition \ref{chinburg} implies that there is a Salem number $\lambda$ such that $E = {\QQ}(\lambda)$. 
Similarly, if $F$ is a Salem field, then $F$ contains a Salem number $\tau$ with $F = {\QQ}(\tau)$. But $F$ is a subfield of $E$, hence $\tau \in E$; by \cite{Sm},
Proposition 3, (ii), this implies that ${\QQ}(\tau) = {\QQ}(\lambda)$, hence $F = E$. 
\qed

\begin{prop} Let $E$ be a Salem field. Then there exists a Salem number $\lambda \in E$ such that the set of Salem numbers in $E$ consists of the
powers of $\lambda$.

\end{prop}

\noindent
{\bf Proof.} This is \cite{Sm}, Proposition 3 (iii). 
\qed

\medskip
The last proposition suggests the following definition

\begin{defn} Let $E$ be a Salem field. The Salem number  $\lambda \in E$ such that the set of Salem numbers in $E$ consists of the
powers of $\lambda$ is called a {\it primitive Salem number}.

\end{defn}

\begin{notation} If $E$ is a Salem field, we denote by $\Lambda_E$ the set of Salem numbers contained in $E$, and by $\lambda_E$ the primitive Salem number of $E$. 

\end{notation}

\subsection{Salem numbers and number theory}\label{Stark and Chinburg}

Let $E$ be a Salem field. We saw in the last section that $E$ is generated by a Salem number, and that $E$ contains a {\it primitive} Salem number
$\lambda_E$ such that the set of Salem numbers $\Lambda_E$ in $E$ consists of the
powers of $\lambda_E$. The aim of this section is to give an explicit expression for $\lambda_E$, in terms of arithmetic invariants of the number field $E$. This is based on work of Stark \cite{Stark} and Chinburg
\cite{Chinburg}. We start by introducing some notation.

\begin{notation} Let $E$ be a Salem field, and let $E_0$ be its maximal totally real subfield. We denote by
$h(E_0)$, respectively $h(E)$ the class numbers of the fields $E_0$, respectively $E$. 
Let $\chi$ be the nontrivial character of ${\rm Gal}(E/E_0)$, and let 
$L(s,\chi)$ be its Artin $L$-function. 
\end{notation}

\begin{theo} 
\label{thm:eps}
Let $E$ be a Salem field of degree $2d$.  Define $u = 2$ if $E$ is generated over $E_0$ by the square root of a unit of $E_0$, and $u = 1$ otherwise. 
Then
$$\epsilon = {\rm exp}(h(E_0)2^{2-d}u L'(0,\chi)/h(E))$$ 

is a Salem number in $E$, and $E = {\bf Q}(\epsilon)$. 

\medskip
The subgroup generated by $\epsilon$ and the units of $E_0$ is of index $2u$ in the group of units of $E$.

\medskip
Moreover, $\lambda_E = \epsilon^{n/2}$ with $n =1$ or $2$.

\end{theo}

\noindent
{\bf Proof.} The fact that $\epsilon$ is a unit of $E$, 
that
$\epsilon > 1$ and that $E = {\bf Q}(\epsilon)$ is Stark's result Theorem 2 in \cite{Stark}, and Chinburg shows that $\epsilon$ is a Salem number, see
\cite{Chinburg}, Theorem 1. In the same paper \cite{Stark}, Stark also proves that subgroup generated by $\epsilon$ and the units of $E_0$ is of index $2u$ in the group of units of $E$.
Finally, Chinburg proves in \cite{Chinburg}, Theorem 1, that if $\sigma$ is a Salem number in $E$, then  $\sigma = \epsilon^{n/2}$ for some positive integer $n$. Applying this to $\sigma = \lambda_E$, we see that either $\lambda_E = \epsilon$ or  $\lambda_E^2 = \epsilon$. 
\qed

\section{Building blocks for signatures}\label{building blocks} 

We will see that signatures play an important role for pseudo-polarized Hodge structures; this is the motivation of the present section.
The results of this section 
will be used in most of the following ones; since they only concern signatures, we work here with \'etale
algebras over $\RR$.

\medskip

Let $E$ be an \'etale algebra over the real numbers, so $E\cong\RR^s\times\CC^t$,
with an $\bf R$-linear involution, denoted by $x \mapsto \overline x$;
the involution can be trivial or not.
For all $a \in E$ such that $\overline a = a$, we denote by $q_a$ the quadratic form $q_a : E \times E \to {\bf R}$ given by
$q_a(x,y) = Tr_{E/{\bf R}}(ax \overline y)$. The aim of this section is to determine the possible signatures of these forms. We denote by $H$ the hyperbolic plane over $\bf R$.

\medskip
The algebra with involution $E$ is the product of the following ``building blocks'', indecomposable algebras with involution:

\begin{enumerate}
\item[(1a)]
$E = {\bf R}$ with trivial involution.

\smallskip
\item[(1b)] $E = {\bf C}$ with trivial involution.
\smallskip

\item[(2a)]
$E = {\bf R}\times {\bf R}$, the involution exchanging the two factors.
\smallskip

\item[(2b)]
$E = {\bf C}\times {\bf C}$, the involution exchanging the two factors.
\smallskip
\item[(2c)]
$E = {\bf C}$, 
and the 
involution is the complex conjugation.
\end{enumerate}

\smallskip

We now determine the possible quadratic forms $q_a$ in all the above cases, and their signatures.

\smallskip

\begin{enumerate}
\item[(1a)]
 $E = {\bf R}$ with trivial involution. Then $q_a \simeq \langle 1 \rangle$ or $q_a \simeq \langle -1 \rangle$,
 depending on whether $a\in \RR$ is positive or negative, with signature $(1,0)$ resp.\ $(0,1)$.

\smallskip
\item[(1b)]
 $E = {\bf C}$ with trivial involution. Then $q_a \simeq H$, hence its signature is $(1,1)$.

\smallskip
\item[(2a)]
 $E = {\bf R}\times {\bf R}$, the involution exchanging the two factors. Then $q_a \simeq H$,  hence its signature is $(1,1)$.

\smallskip
\item[(2b)] $E = {\bf C}\times {\bf C}$, the involution exchanging the two factors.Then $q_a \simeq H \times H$,  hence its signature is $(2,2)$.

\smallskip
\item[(2c)] $E = {\bf C}$, and the involution is the complex conjugation.
Then $q_a \simeq \langle 1,1 \rangle$ or $q_a \simeq \langle -1, -1 \rangle$, 
 depending on whether $a$ is positive or negative;
 the signature is $(2,0)$ resp.\ $(0,2)$.
\end{enumerate}

\medskip
Since quadratic and hermitian forms over $E$ are diagonalizable, the above computations
determine the signature of ${\rm T}(W)$ for any quadratic or hermitian form $W$ over $E$.

\medskip

We now relate cases (2a) and (2c) above to Salem fields (see Definition \ref{def:Salem'} for the notion of Salem field): 

\begin{lemma}
\label{lem:Salem-crit}
Let $E$ be a number field of degree $2d\geq 4$.
Then $E$ is Salem if and only if $E_\RR \cong \RR^2\times \CC^{d-1}$
and there is a $\QQ$-linear involution on $E$ whose extension to $E_\RR$
exchanges the two copies of $\RR$ and acts as complex conjugation of each copy of $\CC$.
\end{lemma}

\begin{proof}
Let $E$ be a Salem field,  with non-trivial $\bf Q$-linear involution $\iota$
and 
fixed field $E_0$. 
By definition, $E_0$ is totally real. The extension $E/E_0$ is of degree 2, hence there exists
$\alpha \in E_0$ such that $E = E_0 (\sqrt \alpha)$. 
The field $E$ has exactly two real embeddings, the other embeddings are complex. This
implies that exactly one of the real embeddings of $E_0$ extends to (a pair of) real embeddings of $E$
while the others extend to (pairs of) complex embeddings.
Hence
$\alpha$ is positive at exactly one embedding of $E_0$, the one that extends to the pair of real embeddings, 
which are thus interchanged by $\iota$.
At all the other embeddings, $\alpha$ is negative 
whence $\iota$ acts as complex conjugation on each embedding.
The base extension to $E_\RR$ thus yields the claim.

For the converse direction, denote the involution by $\iota$ and the fixed field by $E_0=E^\iota$.
Then
$$
(E_0)_\RR = (E_\RR)^{\iota_\RR} = \RR^d$$
shows that $E_0$ is totally real.
Hence $E$ is a Salem field by definition.
\qed
\end{proof}

\section{Signatures}\label{signatures}

We keep the notation of the previous sections. In this paper, we have essentially two cases of interest (RM and SM),
but for clarity and reference we also include the CM case:

\medskip

$\bullet$ $E = E_0$ is totally real (the RM case).

\medskip
$\bullet$ $E$ a CM field, that is, $E_0$ is totally real, $E$ is totally imaginary, and $e \mapsto \overline e$ is the complex conjugation (the CM case).

\medskip
$\bullet$ $E$ is a Salem field, $E_0$ is totally real (the SM case).

\medskip
Let $W$ be a hermitian form over $E$ with respect to the involution $e \mapsto \overline e$; note that when $E = E_0$, this is a quadratic form,
since the involution is trivial.
We focus on the case $W=\langle a\rangle$ for $a\in E_0^\times$ and consider $\rT_a = \rT(W)$.
 To determine its signature, we look at $E \otimes_\QQ \RR$, as an \'etale algebra with involution.
 It is a product of factors $\RR$ and $\CC$,
 indexed by the real embeddings of $E$
 (coming in pairs above some real embeddings of $E_0$)
 and by pairs of the complex embeddings of $E$
 (above the remaining real embeddings of $E_0$);
 then $\rT_a$ decomposes as an orthogonal direct sum according to these factors
 as follows:

\medskip
$\bullet$ If $E$ is totally real, then all the factors are equal to $\RR$, the involution is the identity.
$\rT_a$ is either $1$ or $-1$ on each of these factors, and their number is
$a^+$ respectively $a^-$,
where $a^+$ is the number of real embeddings of $E$ where $a$ is positive, and $a^-$ is the number of real embeddings where
$a$ is negative.
The signature of $\rT_a$ is thus $(a^+,a^-)$.

\medskip
$\bullet$ If $E$ is a CM field,
then all the factors are equal to $\CC$,
the involution is the complex conjugation on each of them.
$\rT_a$ decomposes into an orthogonal sum of binary forms, each of them positive definite or negative definite.
The number of positive definite ones is $a^+$, the number of negative definite ones is $a^-$
where, again, $a^+$ is the number of real embeddings of $E_0$ where $a$ is positive,
and $a^-$ is the number of real embeddings of $E_0$ where
$a$ is negative.
The signature of $\rT_a$ is therefore $(2a^+,2a^-)$.


\medskip
$\bullet$ If $E$ is a Salem field with involution,
then there are two factors $\RR$, interchanged by the involution.
The other factors are equal to $\CC$, the involution is the complex conjugation on each of them.
The first factor $\RR \times\RR$ comes from the real place of $E_0$ that splits into two real places of $E$.
On this factor, the form $\rT_a$ is hyperbolic, independently of the value of $a$.
On the factors isomorphic to $\CC$ the situation is exactly as in the CM case.

\medskip
To conclude, there is  a factor of the hyperbolic plane $H$,
and the number of positive definite spaces is $a^+$,
where $a^+$ is the number of embeddings of $E_0$ that extend to  complex embeddings of $E$ and where $a$ is positive;
along the same lines,
the number of negative definite spaces is $a^-$,
the number of embeddings of $E_0$ that extend to  complex embeddings of $E$ and where $a$ is negative.
The signature of $\rT_a$ is hence $(1+2a^+,1+2a^-)$.

%

\medskip In summary, the above considerations give us the signature of the forms $\rT_a$, hence of the
transfer of any one-dimensional quadratic or hermitian form over $E$.
Since quadratic and hermitian forms over $E$ are diagonalizable, this
determines the signature of ${\rm T}(W)$ for any hermitian form $W$ over $E$.

\section{Endomorphism fields of pseudo-polarized Hodge structures}
\label{proof}

We resume our study of the pseudo-polarized Hodge structures from Section \ref{hodge pp}. The first aim of this section is to
characterize the endomorphism fields of these Hodge structures; this is given by Theorem \ref{thm:U^0} below. 

\medskip Recall that if $T_\QQ$  is a simple pseudo-polarized Hodge structure with endomorphism field $E$, then $T_\QQ$  is a vector space over $E$; set $m=\dim_E T_\QQ$. 
We also need to know what values of $m$ are possible; this
information is provided by  Theorem \ref{thm:U^0} and Proposition \ref{prop:RM-SM}.

\medskip

Finally, at the end of the section we show that Theorem \ref{thm:U^0} and Proposition \ref{prop:RM-SM}     lead to a proof of Theorem \ref{endos}. 

\medskip
As in the previous sections, $E$ is a number field with a $\bf Q$-linear involution $e \mapsto \overline e$, and $E_{\bf R} = E \otimes_{\bf Q} {\bf R}$.
Recall that $E\cong \QQ[x]/(f)$ for an irreducible polynomial $f$ and thus
$$
E_\RR\,\cong\,\RR[x]/(f_1\cdots f_s)\,\cong\, \prod_i\RR[x]/(f_i)
\,\cong\,\RR^a\times\CC^b~,
$$
where the $f_i\in\RR[x]$ are irreducible and there are $a$, $b$ factors of degree one and two respectively. The decomposition is indexed by irreducible factors of $f$, equivalently,
by the $a$ real embeddings and $b$ conjugate pairs of complex embeddings of $E$.

\begin{theo}\label{thm:U^0} Let $T_\QQ$ be a simple pseudo-polarized Hodge structure and
let $E:=A_T$ be its algebra of Hodge endomorphisms with
the involution induced by the adjoint map.  Let $m:=\dim_E T_\QQ$.

\medskip
Then the  possibilities for $E$ and the involution are as follows:

\begin{itemize}
\item[$\bullet$]
$E$ is totally real, the involution is trivial, and $m\geq 2$.

\smallskip
\item[$\bullet$] $E$ is a Salem field, the involution is non-trivial, and $m = 1$.

\end{itemize}
\end{theo}

\noindent
{\bf Proof.}
The real vector space $T_\RR:=T_\QQ\otimes_\QQ\RR$ has an orthogonal
decomposition with $E_\RR$-stable summands:
$$
T_\RR\,=\,T^0\,\oplus\,T^1,\qquad T^0\,:=\,(T^{2,0}\oplus T^{0,2})\cap T_\RR,\quad
T^1\,:=\,(T^0)^\perp~.
$$
The summand $T^0$ is positive definite and as $T_\RR$ has signature $(3,r-3)$, the summand
$T^1$ has signature $(1,r-3)$.

\medskip
Let $E_0$ be the fixed field of the involution; we have either $E = E_0$, in the case of a trivial involution, or $E$ is a quadratic extension of $E_0$.
By Lemma \ref{transfer lemma}, there exists a quadratic or hermitian form $W$ over $E_{\bf R}$
such that $T_{\bf R} = \TT(W)$.

\medskip
Suppose first that $E \not = E_0$, i.e.\ the involution is non-trivial.
Since the action of $E$ on $T^{2,0},T^{0,2}$ is given by the complex embeddings $\delta$,
$\overline{\delta}$ and since $\delta(\bar e)=\overline{\delta(e)}$, we see that this
conjugate pair of embeddings corresponds to a factor $\CC$ of $E_\RR$ and the
involution is the complex conjugation.
Further, $E_{\bf R}$ cannot contain any factor of the type $\bf C \times \bf C$, the involution exchanging the factors, since this would imply a factor $H \oplus H$ of $T^1\subset \TT(W)$
with signature $(2,2)$,
contradicting the fact that the signature of $T^1$ is $(1,r-3)$.
By Section \ref{building blocks}, cases (2a) -- (2c), the possible factors of
$T^1$ thus are
\begin{itemize}
\item
${\bf R} \times {\bf R}$ with involution exchanging the two factors and signature $(1,1)$;
\vspace{-.35cm}
\item
$\bf C$, the involution is complex conjugation and the signature is $(0,2)$.
\end{itemize}
Note that this implies that $E_0$ is totally real.
Moreover, $E_\RR$ admits exactly one factor ${\bf R} \times {\bf R}$,
so $E$ is Salem by Lemma \ref{lem:Salem-crit}.
Even more, $T_\RR$ has exactly one factor ${\bf R} \times {\bf R}$, so $m=1$ as stated.

\medskip
Assume now that $E = E_0$. Then $\delta=\overline{\delta}$ on $E$ and thus the action of $E_\RR$ on $T^0\cong\RR^2$ factors over the summand
$\RR$ of $E_\RR$ corresponding to the real embedding $\delta$ of $E$.
This implies that $m\geq 2$.
Since $T^1$, which has signature $(1,r-3)$, is  $E_\RR$-stable,
we see that $E_\RR$ can have at most one factor $\CC$ with trivial involution
cf.\ the  cases (1a), (1b) from Section \ref{building blocks}.
This leads to two possibilities for $E_\RR$:
\begin{itemize}
\item
either
$E_{\bf R}$ is a product of factors $\bf R$ with trivial involution, in which case $E$ is totally real,
\smallskip
\item
or
 $E_{\bf R}$ could have exactly one factor $\bf C$ with
trivial involution, all the other factors being $\bf R$ with trivial involution.
\end{itemize}

Note that the second alternative  implies
that $m =\dim_E(T_\QQ)= 1$ as $T_\RR$ admits exactly one factor $\CC$ with trivial involution.
This contradicts $m\geq 2$ and hence $E$ is totally real, with trivial involution.
\qed

\medskip
Now we exclude, similar to projective K3 surfaces with RM in \cite[Lemma 3.2]{G}, 
the case that $E$ is totally real and $m=2$.

\begin{prop} 
\label{prop:RM-SM}
Let $T_\QQ$ be a simple pseudo-polarized Hodge structure and
let $E=A_T$ be its algebra of Hodge endomorphisms with
the involution induced by the adjoint map. Set $m =\dim_E T_\QQ$.
If $E$ is totally real, then $m \geqslant 3$.
\end{prop}
\noindent
{\bf Proof.}
By Theorem \ref{thm:U^0}, we have $m\geqslant 2$,
so let us assume $m=2$ in order to establish a contradiction.
Let $G:=SO_E(T_\QQ)(\CC)$ be the subgroup of $SO(T_\QQ)(\CC)$  of elements commuting with the action of $E$.
Let $F\subset \End_\QQ(T_\QQ)$ be the subalgebra of endomorphisms commuting with $G$.

Since $E$ and $G$ commute, the group $G$ maps the eigenspaces of $E$ in the complexification $T_\CC$ into themselves. The eigenspace decomposition is $T_\CC\cong\oplus_\sigma (\CC^2)_\sigma$, where $\sigma$ runs over the real embeddings of $E$.
Since the quadratic form on $T_\QQ$ is non-degenerate, $G$ is a product of copies
of ${\rm SO}(2,\CC)$ acting on the $(\CC^2)_\sigma$; in fact, we can choose a basis $x_\sigma, y_\sigma$
of each $(\CC^2)_\sigma$ such that the quadratic form is given by $x_\sigma y_\sigma$
and ${\rm SO}(2,\CC)\cong\CC^*$ acts as $t\cdot(x_\sigma,y_\sigma)=(tx_\sigma,t^{-1}y_\sigma)$.
The algebra commuting with this action on $(\CC^2)_\sigma$ is $\CC\times\CC$ acting by
$(r,s)\cdot (x_\sigma,y_\sigma)=(rx_\sigma,sy_\sigma)$.

As $F_\CC:=F\otimes_\QQ\CC$ is the algebra of endomorphisms of $T_\CC$ which commute with the action $G$,
we find that  $F_\CC\cong \CC^{2n}$,
where $2n=\dim_\QQ T_\QQ$. Thus 
\begin{eqnarray}
\label{eq:dim's}
\dim_\QQ F=\dim_\QQ T_\QQ.
\end{eqnarray}


The two isotropic subspaces in $(\CC^2)_\delta$, where $\delta$ is the embedding of $E$ defined by the action of $E$ on $T^{2,0}$,
are spanned by $(1,0)$, $(0,1)\in (\CC^2)_\delta$ and thus they are stable under $F$.
These subspaces are $T^{2,0}$ and $T^{0,2}$,
hence $F$ acts by Hodge endomorphisms of $T_\QQ$, so $F\subset A_T$.
It follows that $F$ is a (commutative) field, by Proposition \ref{adjoint} (2), 
and $\dim_FT_\QQ=1$ by \eqref{eq:dim's}. We have $F \subset E = A_T$,  and $T_{\QQ}$ has dimension 1 over $F$ but dimension 2 over $E$, leading to a contradiction.
\qed

\medskip
\noindent
{\bf Proof of Theorem \ref{endos}.}
Let $X$ be  a K3 surface with $a(X)=0$ and consider $T=T_X$,
or more generally a hyperk\"ahler manifold $X$ with pseudo-polarized transcendental Hodge structure $T=T_X$.
Then Theorem \ref{endos} follows directly from Theorem \ref{thm:U^0}
and Proposition \ref{prop:RM-SM}.
\qed

\medskip

\begin{example}
\label{ex:explicit}
Let $U=H \oplus \langle 1,1\rangle$ and $d\in\ZZ_{>1}$, square free, and such that
$d=u^2+v^2$ with $u,v\in\ZZ$.
Then  there is a Salem field $K$ with one-dimensional hermitian form $W'$ such that
$U\cong\TT(W')$; for instance, one can take
$K= {\QQ} (\sqrt{-(u+\sqrt d)})$, $W' =\langle (d+u\sqrt d)/d\rangle$.

\end{example}

\section{From quadratic forms to pseudo-polarized  Hodge structures}
\label{s:period}

In Section \ref{proof}, we have seen that the field of endomorphisms of a simple pseudo-polarized Hodge structure of K3 type is either
a totally real field, or a Salem field, and we also obtained some restrictions on the possible dimensions of such Hodge structures,
as vector spaces over their endomorphism field.

\medskip
The aim of this section is to show that, conversely, starting from certain quadratic forms over totally real fields or hermitian forms over Salem fields, we obtain simple pseudo-polarized Hodge structures. 
This is done in Theorem \ref{thm:pspolRM} below (the totally real case),
and in Theorem \ref{thm:pspolSM} (the Salem field case).

\begin{theo} \label{thm:pspolRM}
Let $E$ be a totally real field of degree $d$ and let $m>2$ be an integer.
Let $W$ be a quadratic form of dimension $m$ over $E$ and let $\sigma:E\hookrightarrow\RR$
an embedding such that the eigenspace $W_\sigma\subset W\otimes_\QQ\RR$ has signature
$(2,m-2)$ or $(3,m-3)$, and such that the signature of $\TT(W)$ is $(3,dm-3)$. 

Then there is an $(m-2)$-dimensional family of  pseudo-polarized K3 type Hodge structures
on $\TT(W)$ which for a very general member in the family is simple and has Hodge endomorphism algebra 
is $E$.
\end{theo}

\noindent
{\bf Proof.}
There is an isomorphism of $E$-vector spaces $\TT(W)\cong E^m$.
The quadratic form on the complexification $W_{\sigma,\CC}\cong \CC^m$ defines a quadric
$$
Q:=\{[\omega]\in\PP(W_{\sigma,\CC}):(\omega,\omega)=0\}
$$
of dimension $m-2\geq 1$
and there is a non-empty open subset, in the analytic topology, $Q^+$
such that $(\omega,\overline{\omega})>0$ for $[\omega]\in Q^+$.
Any such $\omega$ defines a pseudo-polarized Hodge structure of K3 type on $T_\QQ:=\TT(W)$
by
$$
T^{2,0}\,:=\,\CC\omega,\quad T^{0,2}\,:=\,\CC\overline{\omega},\quad
T^{1,1}\,:=\,(T^{2,0}\oplus T^{0,2})^\perp~.
$$
This Hodge structure is not simple if there is a $t\in T_\QQ$ such that $(\omega,t)=0$. So for 
$\omega$ in the complement of the countable union of the hyperplanes $t^\perp$, with $t\in T_\QQ$,
we obtain a simple pseudo-polarized Hodge structure of K3 type.

By construction, $E\subset A_T$. If $E\neq A_T$, then $A_T$ is a number field of degree $dk$ for some $k\geq 2$,
and $\omega$ lies in an eigenspace of
$A_T$ in $T_\CC$, which has dimension $m/k<m$. There are only countably many
subspaces in $\End_\QQ(T_\QQ)$, and hence subalgebras $A_T\subset \End_\QQ(T_\QQ)$,
and thus there are only countably many such (lower-dimensional) complexified eigenspaces contained  in $W_{\sigma,\CC}$.
So for a very general $\omega\in Q^+$ we must have $E=A_T$.
\qed

\

There is an analogous, simpler, result for the Salem field case.

\begin{theo} \label{thm:pspolSM}
Let $E$ be a Salem field of degree $d$ and let $W$ be a one dimensional hermitian form
over $E$ such that $\TT(W)$ has signature $(3,d-3)$.

Then there is a simple pseudo-polarized K3 type Hodge structure
on $\TT(W)$ with endomorphism algebra $E$.

\end{theo}

\noindent
{\bf Proof.}
There is an isomorphism of $E$-vector spaces $\TT(W)\cong E$.
Since $E$ is Salem, $E_\RR\cong \RR^2\times\CC^{(d-2)/2}$,
the involution acts as complex conjugation on each copy of $\CC$
and it interchanges the two copies of $\RR$.

Then there is a summand $U^0$ of $E_\RR$ that is isomorphic to $\CC$
such that the corresponding 2-dimensional subspace of $T_\RR$ is positive definite,
see Section \ref{building blocks}.
This summand corresponds to a pair of conjugate complex embeddings $\sigma,\overline{\sigma}:E\hookrightarrow\CC$.
The corresponding eigenspaces
$U^0_\sigma,U^0_{\overline{\sigma}}\subset U\otimes_\RR\CC$
are isotropic since the adjoint condition \eqref{eq:adj} holds by construction.

We now define a Hodge structure on $T_\QQ:=\TT(W)$ by imposing:
$$
T^{2,0}\,:=\,U^0_\sigma,\quad T^{0,2}\,:=\,\overline{T^{2,0}}\,=\,U^0_{\overline{\sigma}},
\quad T^{1,1}\,:=\,(T^{2,0}\oplus T^{0,2})^\perp~.
$$
One easily verifies that $T_\QQ$ is a simple pseudo-polarized Hodge structure of K3 type as
in Definition \ref{def:pseudo}.
By construction, $E\subset A_T$ and since $\dim_E(T_\QQ)=1$ we must have $E=A_T$.
(Notice that the choice of $T^{2,0}$ is unique up to complex conjugation,
i.e.\ up to interchanging $T^{2,0}$ and $T^{0,2}$.)
\qed

\section{Totally real fields}
\label{s:RM}

We start working towards the proofs of Theorems \ref{RM} and \ref{SM}.
We begin with real multiplication (RM), building on and extending \cite{BGS}.

 If $E$ is  a totally real field, we denote by $\Sigma_E$ be the set of real embeddings of $E$; we have
$E \otimes_{\QQ}\RR = \underset {\sigma \in \Sigma_E} \prod E_{\sigma}$, with $E_{\sigma} = {\bf R}$ for all $\sigma \in \Sigma_E$.  If $W$ is a quadratic form over $E$, then
$W \otimes_{\QQ} \RR$ decomposes as an orthogonal sum $W \otimes_{\QQ} \RR = \underset {\sigma \in \Sigma_E} \oplus W_{\sigma}$;
each of the $W_{\sigma}$ is a quadratic form over $\RR$.

The following result is proved in \cite{BGS}, Corollary 8.2.

\begin{theo}\label{real} Let $V$ be a quadratic form over $\QQ$. Let $E$ be a totally real number field of
degree $d$, let $m$ be an integer with $m \geqslant 1$ such that $m d \leqslant  {\rm dim}(V) - 2$. Let $(r,s)$ be the
signature of $V$, and let $r',s' \geqslant 0$ be integers such that $r' \leqslant r$, $s'\leqslant s$, and $r'+s' = md$.
Suppose that $r' \leqslant m$.

\medskip Then there exists a quadratic form $W$ over $E$ such that the signature of $\TT(W)$ is $(r',s')$ and a quadratic form $V'$ over $\QQ$ such
that $$V \simeq {\TT}(W) \oplus V'.$$

Moreover $W$ can be chosen in such a way that there is an embedding $\sigma: E \to \RR$ with $W_{\sigma}$  of signature $(r',m-r')$.
\end{theo}

This result does not cover the case where $md = {\rm dim}(V) -1$. In order to deal with this case as well, we prove a result that actually gives us more precise
information.

\medskip
The following result is
based on a method of Kr\"uskemper, \cite{K}.

 \begin{theo}\label{new}
 Let $U$ be a quadratic form over $\QQ$ of dimension $r$ and signature $(3,r-3)$.
 Let $E$ be a totally real field of odd degree $d\geq 3$
 and let $m \geqslant 3$ be an integer such that $r = dm$. 
 Fix $r_0\in \{2,3\}$.
Then there exists a quadratic form $W$ over $E$ such that
$$
U \simeq {\TT}(W).
$$
Moreover $W$ can be chosen in such a way that there is an embedding $\sigma: E \to \RR$
with $W_{\sigma}$  of signature $(r_0,m-r_0)$.
\end{theo}

\noindent
{\bf Proof.} 
We first discuss the case $r_0=3$.
Let $\sigma \in \Sigma_E$ and let  $\alpha_1,\dots,\alpha_m \in E^{\times}$ such that $\sigma(\alpha_1) > 0$,
$\sigma(\alpha_2) > 0$, $\sigma(\alpha_3) > 0$, that $\tau(\alpha_1) < 0$ and $\tau(\alpha_2) < 0$, $\tau(\alpha_3) < 0$ for all $\tau \in \Sigma_E$
with $\tau \not = \sigma$, and that $\tau(\alpha_i) < 0$ for  all $i >3$ and   all $\tau \in \Sigma_E$.
Set $W' = \langle \alpha_1,\dots,\alpha_m \rangle$. Note that the signature of $\TT(W')$
is equal to the signature of $U$. This implies that the Witt class of $\TT(W') - U$ is a torsion element of ${\rm Witt}(\QQ)$
(see \cite{BGS}, Theorem 9.3).

\medskip Let us consider the Witt class $X = \TT(W') - U$ in
${\rm Witt}(\QQ)$, and note that $X \in I(\QQ)$,
the fundamental ideal
of ${\rm Witt}(\QQ)$, i.e.\ the
ideal of the even dimensional quadratic forms.
By \cite{K}, Corollary of Lemma 7, page 114, there exists
a torsion form $Y \in I(E)$ such that $\TT(Y) = X$. 
Let $W$ be a quadratic form of dimension $m$ over $E$ representing the Witt class $Y - W'$.
This is possible since $m \geqslant 3$, see \cite {BGS}, Lemma 10.3.
Then $\TT(W) = -U$ in ${\rm Witt}(\QQ)$, 
and since ${\rm dim}(\TT(W)) = {\rm dim}(U)$, we have
$\TT(-W) \simeq  U$.

\medskip

For $r_0=2$, the same arguments go through once we adjust the conditions on $\alpha_3$
such that there is an embedding $\sigma'\in \Sigma_E$ with $\sigma'\neq \sigma$ and $\sigma'(\alpha_3)>0$
while for all other $\tau\in\Sigma_E$ (including $\sigma$) we have $\tau(\alpha_3)<0$.
\qed

\bigskip
The results relevant to the K3 setting are summarized in the following corollary:

\begin{coro}\label{for theorem 1} Let $E$ be a totally real number field of degree $d$ and let $m$ be an integer with $m \geqslant 3$
and $md \leqslant 21$. 
 Fix $r_0\in \{2,3\}$.
Then there exists a quadratic form $W$ over $E$ and a quadratic form $V'$ over $\QQ$ such that
$$
V_{K3} \simeq {\TT}(W) \oplus V'~.
$$
Moreover $W$ can be chosen in such a way that there is an embedding $\sigma: E \to \RR$
with $W_{\sigma}$ of signature $(r_0,m-r_0)$.
\end{coro}

\begin{proof}
If  $md \leqslant 20$, then the result follows from Theorem  \ref{real}
with $r = 3,  r' =r_0$, $s = 19$,  and $s' = md-r'$.

Suppose that
$md = 21$; this implies that $d$ is odd. Let $a \in {\bf Q}$ with $a < 0$, and note that $a$ is represented by $V$ (since $H$ is an orthogonal factor of $V$). Set
$V' = \langle a \rangle$, and
write $V = V' \oplus U$ for some quadratic form $U$ over ${\bf Q}$; the signature of $U$ is then $(3,18)$. By Theorem \ref{new}
there exists a quadratic form $W$ over $E$ with the required properties.
\qed
\end{proof}

\medskip
The only remaining case to cover Theorem \ref{RM}
is $md = 22$, with $m \geqslant 3$; this implies that $d = 2$, hence $E$ is a real quadratic field.

\begin{prop}
\label{md=22} Let $E$ be a real quadratic field and  $r_0\in \{2,3\}$. Then there exists a quadratic form $W$ over $E$ such that
$$
V_{K3} \simeq {\TT}(W)
$$
and that there is an embedding $\sigma: E \to \RR$
with $W_{\sigma}$ of signature $(r_0,11-r_0)$.

\end{prop}

\noindent
{\bf Proof.} Note that $V_{K3} \simeq H \oplus H \oplus H \oplus (-I_{16})$, where $-I_{16}$ denotes the 16-dimensional negative unit form.
Let us write $E = \QQ(\sqrt d)$, for some positive integer $d$. Let $W = \langle \sqrt{d}, \sqrt{d}, (-1)^{r_0+1}\sqrt{d} \rangle \oplus (-I_8)$, as a quadratic form
over $E$. By Example \ref{quadratic example}, we have ${\rm T}(W) \simeq V_{K3}$. Since $\sqrt d$ is positive at one of the embeddings of $E$
and negative at the other, there exists indeed an embedding $\sigma: E \to \RR$
with $W_{\sigma}$ of signature $(r_0,11-r_0)$.
\qed

\medskip
\noindent
{\bf Proof of Theorem \ref{RM}.}
\label{ss:RM}
Let $E$ be a totally real field of degree $d$ and let $m>2$ be an integer so that
$dm\leq 22$. From Corollary \ref{for theorem 1} and Proposition \ref{md=22}
respectively, we know that there exists a quadratic form $W$ over $E$ such that
$\Lambda_{K3}\otimes_{\ZZ} \QQ\cong \TT(W)\oplus V'$
and there is an embedding $\sigma:E\hookrightarrow\RR$ such that the eigenspace $W_\sigma\subset W_\RR$ has 
fixed signature  $(2,m-2)$ or 
$(3,m-3)$.

From Theorem \ref{thm:pspolRM} we know that there exists a
family of dimension $m-2$ of pseudo-polarized Hodge structures
of K3 type on $\TT(W)$ such that the general member has endomorphism algebra $E$.

Let $T:=\Lambda_{K3}\cap \TT(W)$, it is a primitive sublattice of $\Lambda_{K3}$ of rank
$dm$ and signature $(3,md-3)$ and $T_\QQ=\TT(W)$.
The Hodge structure on $T_\QQ$ induces a Hodge structure on $\Lambda_{K3}$ with
$\Lambda_{K3}^{p,q}=T^{p,q}$ for $(p,q)=(2,0)$ and $(0,2)$ and $\Lambda_{K3}^{1,1}$
the orthogonal complement of the direct sum of these two subspaces.
The surjectivity of the period map then implies that there is a (unique by the Torelli theorem)
K3 surface $X$ with $T_X=T$ and thus $A_X=A_T=E$.
\qed

%
%

 \section{Totally real fields of even degree}
 \label{s:even}

 Throughout this section, we let $E$ be a totally real field of degree $d$ and suppose that $d$ is even.
 The aim of this section is to prove an analog of Theorem \ref{new} in the case where $d$ is even; this
 is needed for the treatment of  (totally real) endomorphism algebras of hyperk\"ahler manifolds.

\medskip

We start by recalling some notation
from \cite{BGS}.

\medskip
Let ${\rm N}_{E/{\QQ}} : E \to {\QQ}$ be the norm map;
it induces a homomorphism
$$
{\rm N}_{E/{\QQ}} : E^{\times}/E^{\times 2}  \to {\QQ}^{\times}/({\QQ}^{\times})^2;
$$
let $\Lambda_{E/{\QQ}}$ be the image of this homomorphism.
We denote by $\Lambda^+_{E/{\QQ}}$ the image of totally positive elements of $E$. Further, if $a, b > 0$ are integers such that $d = a +b$, we
denote by $\Lambda^{a,b}_{E/{\QQ}}$ the set
of the elements of $E$ that are positive at $a$ embeddings of $E$, and negative at $b$ embeddings of $E$.

%
%
%
%

%

\begin{theo}\label{even degree}
 Let $m\geqslant 3$ be an  integer, and  let $U$ be a quadratic form over $\QQ$ of dimension $dm$ and signature $(3,dm-3)$.
  Fix $r_0\in \{2,3\}$.
There exists a quadratic form $W$ over $E$ such that
$$
U \simeq {\TT}(W)
$$
and that there is an embedding $\sigma: E \to \RR$
with $W_{\sigma}$  of signature $(r_0,m-r_0)$,
if and only if there  exists
$\alpha \in \Lambda^{1,d-1}_{E/{\QQ}}$
such that
\begin{eqnarray}
\label{eq:cond_N}
{\rm det}(U) = {\mathrm N}_{E/\QQ}(\alpha) \Delta_E^m \ {\rm in} \ {\QQ}^{\times}/({\QQ}^{\times})^2.
\end{eqnarray}
\end{theo}

\noindent{\bf Proof.} 
Let $W$ be as in the theorem and   $\alpha_1,\dots,\alpha_m \in E^{\times}$ be such that $W = \langle \alpha_1,\dots,\alpha_m \rangle$.
Set $\alpha = {\rm det}(W)$; we have
${\rm det}(U) = {\mathrm N}_{E/\QQ}({\rm det}(W)) \Delta_E^m
= {\mathrm N}_{E/\QQ}(\alpha) \Delta_E^m$ in
$\ {\QQ}^{\times}/({\QQ}^{\times})^2$ as in \eqref{eq:cond_N}. 

We claim that $\alpha \in \Lambda^{1,d-1}_{E/{\QQ}}$
or $\alpha \in \Lambda^{d-1,1}_{E/{\QQ}}$. 
Indeed, $\alpha$ has the same sign at all embeddings $\tau\in\Sigma_E$ with $W_\tau$ of signature $(0,m)$ or $(2,m-2)$,
and the opposite sign at the unique embedding with signature $(3,m-3)$ or $(1,m-1)$.
This proves the claim.

Since $d$ is even, the norm of $-\alpha$ is equal to the norm of $\alpha$, hence
we can assume that $\alpha \in \Lambda^{1,d-1}_{E/{\QQ}}$.

\smallskip

In the converse direction, suppose  that there exists $\alpha \in \Lambda^{1,d-1}_{E/{\QQ}}$
such that ${\rm det}(U) = {\mathrm N}_{E/\QQ}(\alpha) \Delta_E^m$
 in
$\ {\QQ}^{\times}/({\QQ}^{\times})^2$,
and assume that $\alpha$ is positive exactly at $\sigma\in\Sigma_E$.
To prove the existence of the quadratic form $W$ with the
required properties, we choose $\alpha_1,\hdots,\alpha_m\in E$ such that
 \begin{itemize}
 \item
 $\alpha_4\hdots,\alpha_m\in\Lambda^{0,d}_{E/\QQ}$ and
 \item
 $\alpha_1,\hdots, \alpha_3\in\Lambda^{1,d-1}_{E/\QQ}$ such that
 $\sigma(\alpha_i)>0$ if and only if $i\leq 2r_0-3$.
 \end{itemize}
 In case $r_0=2$, we moreover assume that $\alpha_2, \alpha_3$ are positive under the same embedding $\sigma'\neq \sigma$.

Set $W' = \langle \alpha_1,\dots,\alpha_m \rangle$.
Let us consider the Witt class $X = \TT(W') - U$ in
${\rm Witt}(\QQ)$, and note that $X \in I(\QQ)$ and that it is torsion (cf.\  \cite{BGS}, Theorem 9.3).

The determinant of $\TT(W')$ is  equal to ${\mathrm N}_{E/\QQ}({\rm det}(W')) \Delta_E^m$
by Lemma \ref{invariants bis}  (ii).
We have ${\rm det}(W') = \alpha_1 \cdots \alpha_m$.
Set $\beta = {\rm det}(W')$, and note that the discriminant
of $X$ is ${\mathrm N}_{E/\QQ}(\alpha \beta) 
= {\mathrm N}_{E/\QQ}(-\alpha \beta)$. 

%
%
By construction, the element $(-1)^{m+1}\alpha \beta$ is totally positive, 
hence
$\disc(X)\in\Lambda^+_{E/{\QQ}}$.
By \cite[ Theorem 9.2]{BGS}, this implies that
there exists a torsion form
$Y \in I(E)$ such that $\TT(Y) \cong X$. Let $W$ be a quadratic form of dimension $m$ over $E$ representing the Witt class $Y - W'$; this is possible
by \cite[ Lemma 10.3]{BGS}.
Then $\TT(W) = -U$
 in ${\rm Witt}(\QQ)$, and since ${\rm dim}(\TT(W)) = {\rm dim}(U)$,
 we have
$\TT(-W) \simeq  U$.
By construction, $-W$
has the claimed signatures for the real embeddings of $E$.
\qed

\medskip
We now apply this result to real quadratic fields.

\begin{coro}\label{d=2}
Let $E$ be a real quadratic field, and let  $U$ be a quadratic form over $\QQ$ of dimension $2m$ and signature $(3,2m-3)$.
Suppose that $m$ is odd and  fix $r_0\in \{2,3\}$.
Then there exists a quadratic form $W$ over $E$ such that
$$
U \simeq {\TT}(W)
$$
and that there is an embedding $\sigma: E \to \RR$
with $W_{\sigma}$ of signature $(r_0,m-r_0)$ if and only if $E$ contains a totally positive element of norm $-\det(U)$.
\end{coro}

\noindent
{\bf Proof.}
 If there exists a quadratic form $W$ with the
required properties, then by Theorem \ref{even degree} there exists $\alpha \in \Lambda^{1,1}_{E/{\QQ}}$ with
${\rm det}(U) =
{\mathrm N}_{E/\QQ}(\alpha) \Delta_E$
 in ${\QQ}^{\times}/({\QQ}^{\times})^2$.
 Recall that $E$ is a  quadratic field, hence $E =  \QQ(\sqrt d)$ for some integer $d > 0$.
 Note that $\sqrt d \in \Lambda^{1,1}_{E/{\QQ}}$.  We have $\Delta_E = -{\mathrm N}_{E/\QQ}(\sqrt d)$ in ${\QQ}^{\times}/({\QQ}^{\times})^2$; therefore
 we have ${\rm det}(U) = -
{\mathrm N}_{E/\QQ}(\alpha \sqrt d)$, and $\alpha \sqrt d$ is totally positive (or totally negative), hence
  $-{\rm det}(U)$ is the norm of a totally positive element of $E$.
  (A priori, this only holds modulo squares, but for quadratic fields this is equivalent to asking it for the actual norm $-{\rm det}(U)$ as in Example \ref{ex:(1,1)}.)

 \medskip Conversely, suppose that there exists a totally positive element $\gamma \in E^{\times}$ of norm $-\det(U)$; set $\beta = \sqrt d$,
 and write $\Delta_E = - {\mathrm N}_{E/\QQ}(\beta) $, as before.
 Set $\alpha = \gamma \beta^{-1}$; we have ${\mathrm N}_{E/\QQ}(\alpha) \Delta_E = \det(U)$ and $\alpha \in  \Lambda^{1,1}_{E/{\QQ}}$,
 so Theorem
 \ref{even degree} implies the existence of the desired quadratic form.
 \qed

 \begin{coro}\label{md=22 new} Let $E$ be a real quadratic field, and let $U$ be a quadratic form
 over $\QQ$ of dimension $2m$ with $m$ odd and signature $(3,2m-3)$; suppose that $U$ has
determinant $-1$. 
 Fix $r_0\in \{2,3\}$.
Then there exists a quadratic form $W$ over $E$ such that
$$
U \simeq {\TT}(W)
$$
and that there is an embedding $\sigma: E \to \RR$
with $W_{\sigma}$ of signature $(r_0,m-r_0)$.
\end{coro}

\noindent
{\bf Proof.} This follows from Corollary \ref{d=2}, since ${\rm det}(U) = -1$.
\qed

\medskip
 Applied to $U=V_{K3}$, this provides a new proof of Proposition \ref{md=22}.

\begin{coro}\label{d=2 m even}
Let $E$ be a real quadratic field, and let  $U$ be a quadratic form over $\QQ$ of dimension $2m$ and signature $(3,2m-3)$.
Suppose that $m$ is even and fix   $r_0\in \{2,3\}$.
Then there exists a quadratic form $W$ over $E$ such that
$$
U \simeq {\TT}(W)
$$
and that there is an embedding $\sigma: E \to \RR$
with $W_{\sigma}$ of signature $(r_0,m-r_0)$ if and only if
$E$ contains an element of norm $\det(U)$.
\end{coro}

\noindent
{\bf Proof.} 
Suppose that there exists a quadratic form $W$ with the required properties. By Theorem \ref{even degree} this implies that
$E$ contains an element $\alpha \in \Lambda^{1,1}_{E/{\QQ}}$  of norm $\det(U)$ in ${\QQ}^{\times}/({\QQ}^{\times})^2$; since
$E$ is a quadratic field, we can assume that the norm of $\alpha$ is $\det(U)$.
Conversely, let $\alpha$ be an element of $E$ with norm ${\rm det}(U)$. Since $\det(U)<0$ by inspection of the signature, we have
$\alpha \in \Lambda^{1,1}_{E/{\QQ}}$,
and hence  Theorem \ref{even degree} implies the existence of the desired quadratic form.
\qed

\section{Real multiplication for non-projective hyperk\"ahler manifolds}
\label{s:HK}

In this section, we start to collect the results analogous to Theorem \ref{RM}, \ref{SM}
for
all known types of hyperk\"ahler manifolds.
We first treat RM and then turn to SM in the next section.
%

Since the surjectivity of the period map has been established for  
all known types of hyperk\"ahler manifolds \cite[Theorem 8.1]{Huy},
the arguments from Section \ref{ss:RM} carry over almost literally
(without the uniqueness statements),
once we provide the required input on the level of quadratic forms
(as prepared in Sections \ref{s:RM}, \ref{s:even}).
Hence we will keep our considerations rather short.

They are organized according to the four known higher dimensional
families of hyperk\"ahler manifolds (cf.\ \cite[p.\ 78]{Rap} or \cite[Table 1]{BGS})
\begin{itemize}
\item
the deformations of the generalized Kummer varieties $K_n(T)$,
where $T$ is a complex 2-dimensional torus,
\item
the deformations of the Hilbert schemes $X^{[n]}$, where $X$
is a K3 surface,
\item
the families of HK manifolds of OG6-type and  OG10-type discovered by O'Grady.
\end{itemize}

The second cohomology group of a HK manifold $X$ carries the
{\it Beauville-Bogomolov-Fujiki form} (BBF-form, for short),
making $H^2(X,\ZZ)$ into a Hodge structure of K3 type. 

We recall the terminology of Oguiso \cite{O 7}, \cite{O 8}, \cite{O 10}
which depends on the restriction of the BBF-form to  the Neron-Severi group.
If this restriction has signature $(1,\rho-1)$, then the manifold
is said to be {\it hyperbolic}, if this restriction is negative definite, it is called {\it elliptic}, and otherwise it is {\it parabolic}. By a result of Huybrechts, we know that
a HK manifold is hyperbolic if and only if it is projective, and the possible Hodge endomorphism fields are determined in \cite{BGS}. For K3 surfaces,
parabolic is equivalent to algebraic dimension $1$, and we have already seen that this case is uninteresting~: the field of Hodge endomorphisms is
$\QQ$. 

\medskip

Therefore, we focus exclusively on {\it elliptic HK manifolds} in what follows.
It is shown in \cite[Theorem 1.4 (1)]{Campana} that the algebraic dimension of such a manifold is zero, and the converse is true in the case of K3 surfaces; however, 
is not clear whether this holds for arbitrary HK manifolds. 
Note that the transcendental lattice of an elliptic HK manifold is indeed pseudo-polarized,
i.e.\ non-degenerate of signature $(3,\cdot)$,
so we can apply the results  developed in the preceding chapters.

%
%
%
%

\subsection{BBF-forms of hyperk\"ahler manifolds}
\label{s:BBF}

To apply our results, we need the 
 list of the BBF-forms of the known HK manifolds, as given in 
\cite[Table 1]{BGS}. However, we only need to know the structure of the {\it rational quadratic forms} obtained from the integral ones after base change to $\QQ$. 
If $X$ is a HK manifold, set $V_X = H^2(X,\QQ)$; with the
notation of Example \ref{HW}, the quadratic forms $V_X$ are as follows:

\begin{itemize}
\item If $X$ is a HK manifold of  $K_n(T)$-type, then $V_X \simeq H^3 \oplus \langle -2k \rangle$,
where $k =  n + 1$, and $n \geqslant 2$.

\smallskip

\item
If $X$ is a K3 surface, then the quadratic form $V_X \simeq V_{K3}$.

\smallskip

\item If $X$ is a HK manifold of  K3$^{[n]}$-type, then  $V_X \simeq V_{K3} \oplus \langle -2k \rangle$,
where $k = n-1$, and $n \geqslant 2$.

\smallskip

\item
If $X$ is a HK manifold of OG6-type, then $V_X \simeq H^3 \oplus \langle -1,-1 \rangle$.

\smallskip

\item
If $X$ is a HK manifold of OG10-type,  then  $V_X \simeq V_{K3} \oplus \langle -2,-6 \rangle$.

\end{itemize}

\subsection{Generalized Kummer manifolds}

\begin{theo}
An elliptic generalized Kummer manifold can only have RM by a real quadratic field.

Given $n>1$ and a real quadratic field $E$, there is
a one-dimensional family of generalized Kummer manifolds $X$ of dimension $2n$
such that a very general member satisfies $A_X=E$.
\end{theo}

\begin{proof}
Since $H^2(X,\QQ) \cong  H^3 \oplus \langle-2n-2\rangle$ has dimension $7$,
Theorem \ref{endos} only allows for real quadratic fields, more precisely $d=2, m=3$.

Conversely, given a real quadratic field $E$,
we set $U = H^3$.
Then Corollary \ref{md=22 new} applies, since $m$ is odd and $\det(U) = -1$.
Therefore the argument from the proof of Theorem \ref{RM} from Section \ref{ss:RM} yields the desired one-dimensional family.
\qed
\end{proof}

\subsection{Hyperk\"ahler manifolds of OG6-type}

\begin{theo}
Elliptic hyperk\"ahler manifolds of {\rm OG6}-type can only have RM by  real quadratic fields.
Let $E$ be a real quadratic field.

\smallskip
There is
a one-dimensional family of elliptic hyperk\"ahler manifolds $X$ of
{\rm OG6}-type such that a very general member satisfies $A_X=E$.

\smallskip
Moreover there is
a two-dimensional family of elliptic hyperk\"ahler manifolds
$X$ of {\rm OG6}-type
such that
a very general member satisfies $A_X=E$
if and only if $E$ contains an element of norm $-1$.

\end{theo}

\begin{proof}
$H^2(X,\QQ) \cong  H^3 \oplus \langle-1,-1\rangle$ having dimension $8$
leaves room for $d=2$ with $m=3, 4$ only.
The case $d = 2$, $m = 3$ is covered by Theorem \ref{real}.
The case $d=2, m=4$ follows from Corollary \ref{d=2 m even}
since $m$ is even and $\det(H^2(X,\QQ)) = -1$.
\qed
\end{proof}

\subsection{Hyperk\"ahler manifolds of K3$^{[n]}$ type}

\begin{theo}
Elliptic hyperk\"ahler manifolds of K3$^{[n]}$ type
can only have RM by totally real number fields of degree $d\leq 7$.

Precisely, let $E$ be a totally real number field of degree $d$ and let $m$ be an integer with $3\leq m \leq 22/d$.
Then there exists an $(m-2)$-dimensional family of
elliptic hyperk\"ahler manifolds of K3$^{[n]}$ type 
such that a very general member $X$ has the properties
$A_X \simeq E$ and ${\rm dim}_E(T_{X,\QQ}) = m$.
%
\end{theo}

\begin{proof}
Let $X$ be an elliptic hyperk\"ahler manifold of K3$^{[n]}$ type ($n>1$).
Then $H^2(X,\QQ) \cong V_{K3} \oplus  \langle-2(n-1)\rangle$.
Since this has dimension $23$ which is prime, we obtain exactly the same endomorphism algebras as for K3 surfaces,
i.e.\ Theorem \ref{RM} carries over.
\qed
\end{proof}

\subsection{Hyperk\"ahler manifolds of OG10-type}

\begin{theo}
Elliptic hyperk\"ahler manifolds of {\rm OG10}-type can admit exactly the following RM structures:
\begin{enumerate}
\item
For $md\leq 22$, same as for K3 surfaces, i.e.\ $A_X$ may be any totally real number fields of degree $d\leq 7$,
and $m$ may range from $3$ to $22/d$.
\item
$md=23$ is impossible.
\end{enumerate}
The case $md=24$ breaks down into the following subcases:
\begin{enumerate}
\item[(3)]
$d = 2, m = 12$ exactly for the real quadratic fields containing an element of norm $-3$;

\item[(4)]
$d = 3, m = 8$ for any totally real cubic field $E$;

\item[(5)]
$d=4, m=6$ exactly for the totally real quartic fields $E$
which  contain an element $a$ with signature $(1,3)$ such that $N(a) = - 3$ modulo squares.

\item[(6)]
$d = 6, m = 4$ exactly for the totally real sextic fields $E$
which  contain an element $a$ with signature $(1,5)$ such that $N(a) = - 3$ modulo squares.

\item[(7)]
$d=8, m=3$ exactly for the totally real octic fields $E$
which  contain an element $a$ with signature $(1,7)$ such that $N(a) = - 3\Delta_E$ modulo squares.
\end{enumerate}
More precisely, each case occurs, 
and the elliptic hyperk\"ahler manifolds deform in $(m-2)$-dimensional families with given RM.
\end{theo}

\begin{proof}
Let $X$ be a hyperk\"ahler manifold of OG10-type.
We have $H^2(X,\QQ)\cong V_{K3} \oplus  \langle-2, -6\rangle$ of dimension $24$.
The cases $md\leq 23$ follow as in the K3$^{[n]}$ case.

For $md=24$, 
the case $d=3, m=8$ is settled by Theorem \ref{new}.
The case $d=2, m=12$ is allowed by Corollary \ref{d=2 m even}
if and only if $E$ contains an element of norm $-3$ as stated.
 The remaining cases follow from Theorem \ref{even degree}.
\qed
\end{proof}


\section{Salem multiplication}\label{Salem multiplication section}

In the third part of the paper, we turn to Salem multiplication, and the proof of Theorem \ref{SM}, as well as its generalizations for HK manifolds. 
%
%
%
%
%
%
%
%
%
%
The unconditional results concerning Salem multiplication for elliptic HK manifolds are as follows: 

\begin{theo}\label{non-maximal theorem} Let $E$ be a Salem field of degree $2d$.

\medskip
{\rm (i)} Suppose that $d \leqslant 3$. Then there exists an elliptic hyperk\"ahler manifold $X$ of $K_n(T)$-type such that $A_X \simeq E$, and an
elliptic hyperk\"ahler manifold $Y$ of {\rm OG6}-type with $A_Y \simeq E$. 

\medskip
{\rm (ii)} Suppose that $d \leqslant 10$. Then there exists a K3 surface $X$ of algebraic dimension $0$ with $A_X \simeq E$.

\medskip
{\rm (iii)} Suppose that $d \leqslant 11$. Then there exists an elliptic hyperk\"ahler manifold $X$ of  {\rm K3}$^{[n]}$-type such that $A_X \simeq E$, and an
elliptic hyperk\"ahler manifold $Y$ of  {\rm OG10}-type with $A_Y \simeq E$.

\end{theo} 

For HK manifolds of OG6-type and $d = 4$, K3 surfaces and $d = 11$, as well as OG10-type and $d = 12$, we furthermore have the following conditional results
which depend on the discriminant $\Delta_E$ of $E$. 

\begin{theo}\label{610} Let $E$ be a Salem field of degree $2d$.

\medskip
{\rm (i)} Suppose that $d =4$. Then there exists an 
elliptic hyperk\"ahler manifold $X$ of  {\rm OG6}-type with $A_X \simeq E$ if and only if $-\Delta_E$ is a square.

\medskip
{\rm (ii)} Suppose that $d = 11$. Then there exists a K3 surface $X$ of algebraic dimension $0$ with $A_X \simeq E$ if and only if
$\Delta_E$ is a square.

\medskip
{\rm (iii)} Suppose that $d =12$. Then there exists an elliptic hyperk\"ahler manifold $X$ of  {\rm OG10}-type  such that $A_X \simeq E$ if and
only if $-3\Delta_E$ is a square.

\end{theo}

\medskip
The above results will be proved in \S \ref{ss:pf}.  
Note that Theorem \ref{non-maximal theorem} (ii) and Theorem \ref{610} (ii) together cover Theorem \ref{SM}.
We start with some first results concerning hermitian forms over Salem fields.

\subsection{Transfers of hermitian forms over Salem fields}\label{transfer Salem}

\medskip To be able to apply Theorem \ref{thm:pspolSM}, we need to characterize the quadratic forms that occur as ${\rm T}(W)$ for some one-dimensional hermitian form $W$. 
Recall that $\Delta_E$ is the discriminant of $E$, and let us denote by $S_E$ the set of prime numbers $p$ such that  $E \otimes_{\QQ} \QQ_p$ is a split algebra,
i.e.\ there is an isomorphism of
${\bf Q}_p$-algebras
$$
E \otimes_{\QQ} \QQ_p  \simeq
E_0 \otimes_{\QQ} \QQ_p \times E_0 \otimes_{\QQ} \QQ_p.
$$

We then say that $E \otimes_{\QQ} \QQ_p$ is a {\it split algebra}. Note that the discriminant of a split algebra is trivial (as an element of $\QQ_p^{\times}/\QQ_p^{\times 2}$).

\begin{theo}\label{realization} Let $E$ be a Salem field of degree $2d$, and let $U$ be a quadratic form of dimension $2d$ over ${\bf Q}$. There exists
a one-dimensional hermitian form $W$ over $E$ such that $U \simeq {\rm T}(W)$ if and only if the following conditions hold:

\begin{enumerate}
\item[{\rm (i)}] ${\rm det}(U) = (-1)^{d}\Delta_E $ in ${\bf Q}^{\times}/{\bf Q}^{\times 2}$.
 \smallskip

\item[{\rm (ii)}]
If $p \in S_E$, then
$U \otimes_{{\bf Q}} {\bf Q}_p \simeq (H \otimes_{{\bf Q}} {\bf Q}_p )^{d}$.
 \smallskip

\item[{\rm (iii)}]
 The signature of $U$ is of the form $(1+2a,1+2b)$ for some integers $a,b \geqslant 0$.
\end{enumerate}
\end{theo}

\noindent  
\noindent
{\bf Proof.} If there exists a hermitian form $W$ over $E$ such that $U \simeq {\rm T}(W)$ then conditions (i), (ii) and (iii) 
follow from  \cite{B 24}, \S 19, determinant condition, hyperbolicity condition, and signature condition (the validity of the signature condition, (iii), is also proved in 
\S \ref{signatures}). 
Conversely, suppose that conditions (i)-(iii) hold. The existence of a hermitian form $W$ over $E$
such that $U \simeq {\rm T}(W)$
follows from \cite{B 24}, Theorem 17.2. Indeed, since $E$ is a field, we have $\sha_E = 0$.
The hypotheses
imply that condition (L 1) holds (see \cite {B 24}, Proposition 19.3). Therefore \cite{B 24}, Theorem 17.2 implies that
there exists a hermitian form $W$ over $E$ such that $U \simeq {\rm T}(W)$.
\qed

\medskip
In  \S \ref{hermitian BF}, we apply this result to  the different pseudo-polarizations of the currently known HK manifolds.

\section{Hermitian forms over Salem fields and BBF-forms}\label{hermitian BF}

Let $E$ be a Salem field of degree $2d$. As we have seen in  \S \ref{s:period},
  the pseudo-polarized Hodge structures
with endomorphism field $E$ and signature $(3,2d-3)$ are transfers of one-dimensional hermitian forms with the properties of Theorem
 \ref{thm:pspolSM}. Moreover, to obtain HK manifolds with this pseudo-polarized Hodge structure, we have to show that it is an orthogonal summand
of the Hodge structure of a HK manifold. In order to do this, we use the structure of the second cohomology groups of HK manifolds together with their associated BBF-forms (see \S \ref{s:BBF}).

\medskip

\subsection{Maximal degree} 

In this subsection and the next one, we combine the method of \S \ref{transfer Salem} 
with the information on BBF-forms from \S\ref{s:BBF},
and we prove the arithmetic results necessary for the proofs of Theorems \ref{non-maximal theorem} and \ref{610}.

\medskip We start with the case where the degree of the Salem field is equal to the dimension of the BBF-form of the HK manifold.
Therefore, the results of this subsection will be useful for K3 surfaces, as well as for HK manifolds of OG6-type and of OG10-type, since the quadratic forms associated to these
manifolds are even dimensional.

\medskip

\begin{theo}\label{41112} Let $E$ be a Salem field of degree $2d$.

\medskip
{\rm (i)} Suppose that $d =4$ and let $U = H^3 \oplus \langle -1,-1 \rangle$. Then there exists a one dimensional hermitian form $W$ over $E$
such that $U \simeq {\TT}(W)$ if and only if $-\Delta_E$ is a square.

\medskip
{\rm (ii)} Suppose that $d = 11$ and let $U = V_{K3}$. Then there exists a one dimensional hermitian form $W$ over $E$
such that $U \simeq {\TT}(W)$ if and only if $\Delta_E$ is a square. 

\medskip
{\rm (iii)} Suppose that $d =12$ and let $U = V_{K3} \oplus \langle -2,-6 \rangle$. Then there exists a one dimensional hermitian form $W$ over $E$
such that $U \simeq {\TT}(W)$ if and only if $-3\Delta_E$ is a square. 

\end{theo}

\noindent
{\bf Proof.} The proof of necessity of the condition on $\Delta_E$ is the same in all three cases. Indeed, suppose that there exists a hermitian form $W$ over $E$ such that $U \simeq {\TT}(W)$. By condition (i) of Theorem \ref {realization},
we have ${\rm det}(U) =(-1)^d \Delta_E$ in ${\bf Q}^{\times}/{\bf Q}^{\times 2}$. We have  ${\rm det}(U) = -1$ in cases (i) and (ii), and 
${\rm det}(U) = -3$ in case (iii), hence this implies that $-\Delta_E$ is a square if $d = 4$, that $\Delta_E$ is a square if $d = 11$ and that
$-3\Delta_E$ is a square if $d = 12$.

\medskip
Conversely, suppose that the condition on $\Delta_E$ holds, and let us prove the existence of a one dimensional hermitian form $W$ over $E$
such that $U \simeq {\TT}(W)$; this is done by applying Theorem \ref{realization}. Condition (iii) of Theorem \ref{realization} trivially holds, and condition (i) is checked as above: we have 
${\rm det}(U) =(-1)^d \Delta_E$ in ${\bf Q}^{\times}/{\bf Q}^{\times 2}$ in all three cases.

\medskip
It remains to check condition (ii). Let $p$ be a prime number such that $p \in S_E$, and note that this implies that $\Delta_E = 1$ in $\QQ_p^{\times}/\QQ_p^{\times 2}$. We have to show that $U \otimes_{{\bf Q}} {\bf Q}_p \simeq (H \otimes_{{\bf Q}} {\bf Q}_p )^{d}$. The two quadratic forms have the same dimension; it remains to prove that they
have the same determinant and Hasse invariant.

\medskip If $d = 4$ or $11$, then ${\rm det}(U) = -1$; for $d = 11$, this immediately implies the equality of the determinants, since 
for $d$ odd, we have ${\rm det}(H^d) = -1$. Suppose that $d = 4$, and note that we have ${\rm det}(U) = -1$ and ${\rm det}(H^4) = 1$.
On the other hand, the hypothesis implies that $\Delta_E = -1$ in $\QQ_p^{\times}/\QQ_p^{\times 2}$, and $\Delta_E = 1$ in $\QQ_p^{\times}/\QQ_p^{\times 2}$ because $p \in S_E$, 
hence $-1$ is a square in ${\QQ}_p$; this implies that ${\rm det}(U) = {\rm det}(H^4)$ in $\QQ_p^{\times}/\QQ_p^{\times 2}$. 

\medskip Let us show the equality of the determinants for $d = 12$. In this case, we have ${\rm det}(U) = -3$.
 Since $p \in S_E$, we have $\Delta_E = 1$ in $\QQ_p^{\times}/\QQ_p^{\times 2}$. By hypothesis $-3\Delta_E$ is a square,
hence this implies that ${\rm det}(U) = 1 = {\rm det}(H^{12})$ in $\QQ_p^{\times}/\QQ_p^{\times 2}$.

\medskip 
We now check that the Hasse invariants coincide; we apply the results of Example \ref{HW} for the computation of these invariants.

\medskip
Suppose first that $d = 4$. We have $w(H^4) = 0$, and $w(U) = w(H^3) + (-1,-1) + (-1,1)$ = $(-1,-1) + (-1,-1) = 0$, therefore $w(H^4) = w(U)$.

\medskip
Assume now that $d = 11$. We have $w(V_{K3}) = (-1,-1)$ and $w(H^{11}) = (-1,-1)$, therefore $w(U) = w(H^{11})$. 

\medskip
Finally, suppose that $d = 12$. 
We have $w(V_{K3}) = (-1,-1)$ and $w(-2,-6) = (-1,-1)$, therefore 
$w(U) = ({\rm det}(V_{K3}),3) = (-1,3)$. On the other hand, $w(H^{12}) = 0$. We have
$\Delta_E = 1$ in $\QQ_p^{\times}/\QQ_p^{\times 2}$ because  $p \in S_E$. By hypothesis, $-3\Delta_E$ is a square, hence we have $3 = -1$ in 
$\QQ_p^{\times}/\QQ_p^{\times 2}$, therefore $(-1,3) = 0$ in ${\rm Br}_2(\QQ_p)$. This implies that $w(U) = w(H^{12}) = 0$.

\medskip In summary, we have checked that the dimension, determinant and Hasse invariant of $U \otimes_{{\bf Q}} {\bf Q}_p $ and 
$(H \otimes_{{\bf Q}} {\bf Q}_p )^{d}$  coincide for all $p \in S_E$,
hence we have $U \otimes_{{\bf Q}} {\bf Q}_p \simeq (H \otimes_{{\bf Q}} {\bf Q}_p )^{d}$ for all $p \in S_E$; this implies that condition (ii)
of Theorem \ref {realization} is satisfied. Therefore there exists a one-dimensional hermitian form $W$ over $E$
such that $U \simeq {\TT}(W)$.
\qed

\subsection{Non-maximal degree}

We now turn to the cases where the degree of the Salem field is smaller than the dimension of
the quadratic form. This is always the case for HK manifolds of  $K_n(T)$-type and K3$^{[n]}$-type, since the associated quadratic forms
are odd dimensional. We start with the results that will be used for these manifolds.

\begin{theo}\label{Kum} Let $E$ be a Salem field of degree $2d$, let $k > 0$ be an integer. 
Suppose that either 

\medskip
$\bullet$ $V = H^3  \oplus \langle -2k \rangle$ and $d \leqslant 3$, or 

\medskip 
$\bullet$ $V = V_{K3} \oplus \langle -2k \rangle$ and $d \leqslant 11$.

\medskip Then there exists
a hermitian form $W$  over $E$ and a quadratic form $V'$ over $\QQ$ such that
$$
V \simeq {\TT}(W) \oplus V'~.
$$


\end{theo}

\begin{proof} Suppose first that $d < 3$, respectively $d < 11$. Let $W$ be a hermitian form over $E$ such that the signature of $\TT(W)$ is $(3,2d-3)$.
Note that ${\rm dim}(V) - {\rm dim}(\TT(W)) > 2$, 
hence  \cite{BGS}, Lemma 2.3 implies the existence of a quadratic form $V'$ with the desired property.

\medskip Assume now 
that $d = 3$, respectively $d = 11$. We refer to Example \ref{HW} for the determination of the Hasse invariants below.

\medskip

We have ${\rm det}(V) = 2k$ and $w(V) = (-1,-1) + (-1,-2k) = (-1,2k)$. Set $h = -2k \Delta_E$.
Since $H$ is an orthogonal factor of $V$, there exists
$x \in V$ such that $q(x,x) = h$; let $U$ be a quadratic form over $\bf Q$ such that $V \simeq U \oplus V'$. We have ${\rm det}(U) =
{\rm det}(V){\rm det}(V') = (2k)(-2k \Delta_E) = -\Delta_E$; this implies that  
$w(V) = w(U) + (-\Delta_E,-2k \Delta_E)$. 

\medskip Since $d$ is odd, the discriminant of $E$ is positive by Lemma \ref{negative};
therefore the signature of $U$ is $(3,2d-3)$.

\medskip

If $p \in S_E$, then $\Delta_E = 1$ in
$\QQ_p^{\times}/\QQ_p^{\times 2}$, hence  the above computation implies that $w(U) = (-1,2k) + (-1,-2k) = (-1,-1)$ at $p$; therefore we have
  $w(U) = w(H^3)$, respectively $w(H^{11})$. This implies that if $p \in S_E$, then $U \otimes_{\bf Q} {\bf Q}_p$
is isomorphic to $H^{3} \otimes _{\QQ} \QQ_p $, respectively $H^{11} \otimes _{\QQ} \QQ_p $.
By Theorem \ref{realization} there exists a hermitian form $W$ over $E$ such that $U \simeq \TT(W)$.
\qed
\end{proof}

\medskip

In the following theorem, we see that the above result can also be applied to HK manifolds of 
OG6-type and OG10-type, when the degree of the Salem field is not maximal.

\begin{theo}\label{O}  Let $E$ be a Salem field of degree $2d$, and
suppose that either 

\medskip
$\bullet$ $V = H^3  \oplus \langle -1,-1 \rangle$ and $d \leqslant 3$, or 

\medskip 
$\bullet$ $V = V_{K3} \oplus \langle -2,-6 \rangle$ and $d \leqslant 11$.

\medskip Then there exists
a hermitian form $W$  over $E$ and a quadratic form $V'$ over $\QQ$ such that
$$
V \simeq {\TT}(W) \oplus V'~.
$$

\end{theo}

\noindent
{\bf Proof.} Set $U = H^3  \oplus \langle -2 \rangle$ in the first case, and $U = V_{K3} \oplus \langle -2 \rangle$
 in the second one. By Theorem 
\ref{Kum}, there exists a hermitian form $W$ over $E$ and a quadratic form $U'$ over $\QQ$ such that 
$U \simeq {\TT}(W) \oplus U'$. 
Setting $V' = U' \oplus \langle -2 \rangle$ in the first case, we get $H^3 \oplus \langle -2,-2 \rangle \simeq {\TT}(W)\oplus V'$. Note that
$\langle -2,-2 \rangle \simeq \langle -1,-1 \rangle$, hence this implies that $V \simeq {\TT}(W) \oplus V'$, as claimed.
In the second case, set
$V' = U' \oplus \langle -6 \rangle$; we obtain $V \simeq {\TT}(W) \oplus V'$ in this case as well.
\qed

\medskip
Finally, we prove a result that will be used for K3 surfaces.

 \begin{theo} \label{Thm:10.4} Let $E$ be a Salem field of degree $2d$, with $2 \leqslant d \leqslant 10$.
 Then there exists a hermitian form $W$ over $E$ such that ${\rm T}(W)$ has signature $(3,2d-3)$, and a quadratic form $V'$ with
 $$V_{K3} \simeq {\rm T}(W) \oplus V'.$$
 \end{theo}

 \noindent
 {\bf Proof.} Set $V = V_{K3}$, and suppose first that $d \not = 10$. Let $W$ be a hermitian form over $E$ such that the signature of ${\rm T}(W)$ is $(3,2d-3)$. By \cite{BGS},
 Lemma 2.3, there exists a quadratic form $V'$ over $\bf Q$ such that $V \simeq {\rm T}(W) \oplus V'.$

 \medskip Assume now that $d = 10$, and set $V' = \langle -1,\Delta_E \rangle$. Since $d = 10$,  Lemma \ref{negative} 
 implies that $\Delta_E$ is negative, hence $V'$ is negative definite. There exists a quadratic form $U$ such that $V \simeq U \oplus V'$ (see
 \cite {BGS}, Lemma 2.3). The signature of $U$ is $(3,17)$, and ${\rm det}(U) = \Delta_E$. Since $d = 10$, condition (i) of
 Theorem \ref{realization} is satisfied. Conditions (iii)  obviously holds, 
 hence it remains to check condition (ii) of Theorem \ref{realization}.
 Let $p \in S_E$. 
 Then we have $\Delta_E = 1$ in ${\bf Q}_p^{\times}/{\bf Q}_p^{\times 2}$, and hence ${\rm det}(U) = 1$ and
 $w(V') = 0$ at $p$. 
 Since $V \simeq U \oplus V'$, we have
 \begin{eqnarray}
 \label{eq:w's}
 w(V) = w(V') + w(U) + ({\rm det}(V'),{\rm det}(U)).
 \end{eqnarray}
 Since ${\rm det}(U) = 1$,  the last term in \eqref{eq:w's} is also trivial, and
 we infer that
  $w(V) = w(U)$. By Example \ref{HW} we have $w(V) = (-1,-1)$ and 
 $w(H^{10}) = (-1,-1)$, therefore $w(U) = w(H^{10})$. 
 
 \medskip In summary, we have checked that if $p \in S_E$, then the invariants of $U$ and $H^{10}$ at $p$ are equal,
 therefore we have
 $U \otimes_{{\bf Q}} {\bf Q}_p \simeq (H \otimes_{{\bf Q}} {\bf Q}_p )^{10}$ for $p \in S_E$. This implies that condition (ii) of Theorem
 \ref{realization} also holds. Therefore there exists a hermitian form $W$ over $E$ such that $U \simeq {\rm T}(W)$.
\qed

\section{Hyperk\"ahler manifolds with prescribed Salem multiplication}\label{prescribed Salem}
\label{s:SM}

In this section, we prove Theorems \ref{non-maximal theorem} and \ref{610}
(and thus also Theorem \ref{SM}). 
Throughout we let $E$ be a Salem field of degree $2d$ and $\Delta_E$ its discriminant.

\subsection{K3 surfaces}

\begin{theo}[Theorem \ref{SM}]
\label{K3} There exists a K3 surface $X$ of algebraic dimension 0 such that $A_X \simeq E$ if and only if either $d \leqslant 10$, or
$d = 11$ and $\Delta_E$ is a square. 

\end{theo}

\noindent
{\bf Proof.} Theorem \ref{41112} (ii) and Theorem \ref{Thm:10.4}  imply that if the conditions on $d$ and $\Delta_E$ are satisfied, then 
there exists a hermitian form $W$ over $E$ such that
$\Lambda_{K3}\otimes_{\ZZ} \QQ\cong \TT(W)\oplus V'$ for some
negative definite quadratic form $V'$ (for $d \leqslant 10$) or $\Lambda_{K3}\otimes_{\ZZ} \QQ\cong \TT(W)$ (when $d = 11$). 
By Theorem \ref{thm:pspolSM}, there exists a pseudo-polarized Hodge structure
of K3 type with endomorphism algebra $E$ on $\TT(W)$.
Let $T=\Lambda_{K3}\cap \TT(W)$;  it is a primitive sublattice of $\Lambda_{K3}$ of rank
$2d$ and signature $(3,2d-3)$.
By way of the inclusion $T\subset\Lambda_{K3}$, the Hodge structure on $T_\QQ=\TT(W)$ induces a Hodge structure on $\Lambda_{K3}$ with
$\Lambda_{K3}^{p,q}=T^{p,q}$ for $(p,q)=(2,0)$ and $(0,2)$ and $\Lambda_{K3}^{1,1}$
the orthogonal complement of the direct sum of these two subspaces.
The surjectivity of the period map then implies that there is a (unique by the Torelli theorem)
K3 surface $X$ with $T_X=T$ and thus $A_X=A_T=E$.
\qed

\subsection{Generalized Kummer manifolds}

\begin{theo} Suppose that $d \leqslant 3$ and let $n \geqslant 2$ be an integer. Then there exists an elliptic generalized Kummer manifold $X$ of dimension $2n$ 
such that $A_X \simeq E$. 

\end{theo}

\noindent
{\bf Proof.} Recall from \S \ref{s:BBF}  that if  $X$ is of $K_n(T)$-type, then  $V_X \simeq H^3 \oplus \langle -2k \rangle$,
where $k =  n + 1$, and $n \geqslant 2$. Theorem \ref{Kum} implies that there exists a hermitian form $W$ over $E$ such that
$V_X \cong \TT(W)\oplus V'$ for some
negative definite quadratic form $V'$. The same argument as in the proof of Theorem \ref{K3} shows that there 
exists a pseudo-polarized Hodge structure
of K3 type with endomorphism algebra $E$ on $\TT(W)$, and by the surjectivity of the period map, we obtain a HK manifold $X$
of $K_n(T)$-type with $A_X \simeq E$. 
\qed

\subsection {Hyperk\"ahler manifolds of   {\rm OG6}-type} 

\begin{theo} Suppose that either $d \leqslant 3$ or $d = 4$ and $-\Delta_E$ is a square.
Then there exists an elliptic HK manifold $X$ of OG6-type such that $A_X \simeq E$.

\end{theo}

\noindent
{\bf Proof.} If $X$ is of   {\rm OG6}-type, 
then $V_X \simeq H^3 \oplus \langle -1,-1 \rangle$. Theorem \ref{41112} (i) and Theorem \ref{O} imply that
if the hypotheses on $d$ and $\Delta_E$ are satisfied, then there exists a hermitian form $W$ over $E$ such that
$V_X \cong \TT(W)\oplus V'$ for some
negative definite quadratic form $V'$ (if $d \leqslant 3$) or a hermitian form $W$ over $E$ such that
$V_X \cong \TT(W)$ (when $d = 4$). Applying Theorem \ref{thm:pspolSM} and the surjectivity of the period map as in the previous
theorems, we obtain a HK manifold $X$ of   {\rm OG6}-type such that  $A_X \simeq E$. 
\qed

\subsection {Hyperk\"ahler manifolds of   K3$^{[n]}$-type}

\begin{theo} Suppose that $d \leqslant 11$ and let $n \geqslant 2$ be an integer. Then there exists an elliptic HK manifold $X$ of {\rm K3}$^{[n]}$-type
such that $A_X \simeq E$. 

\end{theo}

\noindent
{\bf Proof.} For a HK manifold $X$ of  K3$^{[n]}$-type, we have $V_X \simeq V_{K3} \oplus \langle -2k \rangle$,
where $k = n-1$. There exists a hermitian form $W$ over $E$ such that $V_X \cong \TT(W)\oplus V'$ for some
negative definite quadratic form $V'$; this follows from  Theorem \ref{Kum}. As in the other cases, we apply 
Theorem \ref{thm:pspolSM} and the surjectivity of the period map to obtain a HK manifold of type K3$^{[n]}$ with $A_X \simeq E$. 
\qed

\subsection {Hyperk\"ahler manifolds of   {\rm OG10}-type} 

\begin{theo} 
\label{thm:OG10}
Suppose that either $d \leqslant 11$ or $d = 12$ and $-3\Delta_E$ is a square.
Then there exists an elliptic HK manifold $X$ of   {\rm OG10}-type  such that $A_X \simeq E$.

\end{theo}

\noindent
{\bf Proof.}  If $X$ is of   {\rm OG10}-type, then $V_X \simeq  V_{K3} \oplus \langle -2,-6 \rangle$. By Theorem \ref{41112} (iii) and Theorem \ref{O}, we see that 
if the hypotheses on $d$ and $\Delta_E$ are satisfied, then there exists a hermitian form $W$ over $E$ such that
$V_X \cong \TT(W)\oplus V'$ for some
negative definite quadratic form $V'$ (if $d \leqslant 11$) or a hermitian form $W$ over $E$ such that
$V_X \cong \TT(W)$ (when $d = 12$). Applying Theorem \ref{thm:pspolSM} and the surjectivity of the period map as previously
provides a HK manifold $X$ of   {\rm OG10}-type such that  $A_X \simeq E$. 
\qed

\subsection{Proof of Theorems \ref{non-maximal theorem}, \ref{610}}
\label{ss:pf}

Theorems \ref{non-maximal theorem}, \ref{610} follow by collecting the results of Theorems \ref{K3} -- \ref{thm:OG10}
and sorting them by the degree of $E$ being maximal (Theorem \ref{610}) or not (Theorem  \ref{non-maximal theorem}).
\qed

\section{Salem fields and dynamics}
\label{s:dyn}

Finally we turn to the applications of our findings in complex dynamics.
In particular, we prove Theorem \ref{thm:bir} from the introduction (Corollary \ref{finite Bir}).

\subsection{Bimeromorphic automorphisms of hyperk\"ahler manifolds and dynamical degrees}\label{bimer}

 \medskip
Let $X$ be a hyperk\"ahler manifold,  and let $a: X \to X$ be a bimeromorphic automorphism;
it induces an isomorphism $a^* : H^2(X,{\bf Z}) \to H^2(X,{\bf Z})$.
The following result is due to Oguiso \cite{O 10}, generalizing a theorem of McMullen \cite{Mc1}.

\begin{theo}{\rm (}\cite{O 10}, {\rm Theorem 2.4)}
\label{thm:McM}
The characteristic polynomial of $a^*$ is a product of at most one Salem polynomial or  reciprocal quadratic polynomial and  finitely many cyclotomic polynomials.
\end{theo}

The {\it dynamical degree} of
$a$ is by definition the spectral radius of $a^*$, i.e.\ the maximum of the absolute values of the eigenvalues of $a^*$;
hence it is either 1, a reciprocal quadratic integer, or a Salem number. This motivates the following question, raised by
McMullen for K3 surfaces (see \cite {Mc1}, \cite{Mc2}, \cite{Mc3}~):

\begin{question}\label{dynamical degree}
Which Salem numbers occur as dynamical degrees of automorphisms of complex hyperk\"ahler manifolds~?
\end{question}

This is equivalent to the following question

\begin{question}\label{oldquestion}
Which Salem polynomials occur as a factor of the characteristic polynomial of $a^*$, where $a$ is an automorphism of a hyperk\"ahler manifold~?
\end{question}

\medskip From the point of view of the present paper, it is natural to seek for a link with Hodge endomorphisms, and Salem multiplication. Indeed, if
 $a: X \to X$ is a bimeromorphic automorphism of a HK manifold $X$, then $a^* : H^2(X,{\bf Z}) \to H^2(X,{\bf Z})$ preserves the Hodge structure and the intersection form, and hence it restricts to  an automorphism $a^* : T_X \to T_X$; therefore $a$ induces a Hodge endomorphism of $X$. 
 
 \medskip
 The work of Oguiso allows us to analyze the possibilities for such Hodge endomorphisms in the elliptic, hyperbolic and parabolic case. We have the following:
 
 \medskip
 $\bullet$ If $X$ is hyperbolic (i.e. projective) then the characteristic polynomial of $a^*|T_X$ is a power of a cyclotomic polynomial; 
 
 \medskip
 $\bullet$ If $X$ is parabolic, then $a^* : T_X \to T_X$ is the identity. 
 
 \medskip
 $\bullet$ If $X$ is elliptic, then either $a^* : T_X \to T_X$ is the identity, or the characteristic polynomial of $a^*|T_X$ is a Salem polynomial.
 
 \medskip
 This is proved in \cite[ Theorem 2.4]{O 8}, thus predating Theorem \ref{thm:McM}.
 
 \bigskip This shows that the connection between bimeromorphic automorphisms $a$ such that the characteristic polynomial of $a^*|_{T_X}$ has a Salem polynomial as a factor, and Salem multiplication only exists in the elliptic case.
 
 \medskip
 Of course, Question \ref{oldquestion} is of interest in the other cases too. In particular, 
 McMullen realized several Salem numbers as the dynamical degree of an automorphism of a {\it projective}  K3 surface (see \cite{Mc3}).  For such an automorphism $a : X \to X$, the Salem polynomial divides the characteristic polynomial of $a^*|{\rm Pic}(X)$, whereas the
 characteristic polynomial of  $a^*|T_X$ is a power of a cyclotomic polynomial. Moreover, in the case of projective K3 surfaces, reciprocal quadratic integers can also occur. These examples are very interesting, but we will not discuss
 them further in this paper. 
 
 \medskip 
 On the other hand, if $a : X \to X$ is a bimeromorphic automorphism of a {\it parabolic} HK manifold $X$, then the dynamical degree of $a$ is equal to one (see \cite{O 7}, Remark 3.1 (2)). 
  
 \medskip
 In the following, we only consider {\it elliptic} HK manifolds. We have the following
 
 \begin{lemma}
\label{lem:a*}
Let $X$ be an elliptic hyperk\"ahler manifold,  let $S$ be a Salem polynomial with Salem number $\lambda$, and set $E = {\QQ}[x]/(S)$.
Suppose that $X$ has a bimeromorphic automorphism with dynamical degree $\lambda$. Then $A_X \simeq E$.
\end{lemma}

\noindent
{\bf Proof.} Indeed, let $a$ be such an automorphism; then the induced isomorphism $a^* : H^2(X,{\bf Z}) \to H^2(X,{\bf Z})$ preserves the Hodge structure and
the intersection form. This implies that $a^*$ induces a Hodge endomorphism.
As $\Pic(X)$ is negative definite, the group ${\rm Bir}(X)$ of bimeromorphic automorphisms acts through a finite group on $\Pic(X)$.
This implies that $E\subseteq A_X$,
which already suffices to deduce the claimed equality  $E=A_X$.
Indeed,  Theorem \ref{endos} shows that $A_X\supseteq E$ can only be a Salem field.
But then, combining this with Proposition \ref{prop:E=F} proves that $E=A_X$.
\qed
%
%

\medskip
This suggests the following question, an analog of Question \ref{oldquestion}:

\begin{question}\label{newquestion}
Which Salem fields occur as Hodge endomorphism algebras of elliptic hyperk\"ahler manifolds~?
\end{question}

The answer to this question is given in  Theorems \ref{non-maximal theorem} and \ref{610}.
  We now turn to Question \ref{oldquestion}, and we start by a definition:

\begin{defn} Let $S$ be a Salem polynomial, and let $X$ be an elliptic hyperk\"ahler manifold. We say that an automorphism $a$ of $X$
is a {\it Salem automorphism} with polynomial $S$ if the characteristic polynomial of the restriction of $a^*$ to $T_X$ is $S$. 

\end{defn}

In the following subsection, we summarize the results concerning K3 surfaces. 

\subsection{K3 surfaces}

Recall that a two dimensional HK manifold is a K3 surface, and that in this case, elliptic is equivalent to algebraic dimension 0. We will
use the latter terminology, since ``elliptic K3 surface'' has a different meaning.





\begin{coro}\label{one square}  Let $S$ be a Salem polynomial of degree $d$, and set $E = {\bf Q}[x]/(S)$. Then there exists a
K3 surface $X$ of algebraic dimension $0$ such that
$A_X \simeq E$ if and only if either $d \leqslant 20$, or $d = 22$ and $|S(1)S(-1)|$
is a square.

\end{coro}

\noindent
{\bf Proof.} This follows from Theorem \ref{K3}, since  the discriminant of $E$ is equal to $(-1)^{d/2}S(1)S(-1)$
in ${\bf Q}^{\times}/{\bf Q}^{\times 2}$,  cf.\ Lemma \ref{Salem discriminant}.
The sign of $\Delta_E$, or equivalently of  $(-1)^{d/2}S(1)S(-1)$, is positive by Lemma \ref{negative},
whence it suffices to consider $|S(1)S(-1)|$.
\qed

\medskip
The answer to Question \ref{oldquestion} for Salem polynomials of degree $22$ is also known. We have the following result, that provides
an answer to a question of Gross and McMullen in \cite{GM}:

\begin{theo}\label{all squares} Let $S$ be a Salem polynomial of degree $22$. 
Then there exists a K3 surface $X$ of algebraic dimension $0$
having a Salem automorphism with polynomial $S$ if and only if  $|S(1)|$, $|S(-1)|$ and $|S(1)S(-1)|$ are all squares.

\end{theo}

\noindent{\bf Proof.} See \cite{BT}, Corollary of Theorem A; in the special case where $|S(1)S(-1)| =1$, this is also proved in \cite{GM},
Theorem 1.7.
\qed

\medskip
Note that for Salem polynomials $S$ of degree $22$, there is a difference between the conditions of Corollary \ref{one square} and Theorem 
\ref{all squares}: in both cases, we ask that $|S(1)S(-1)|$ is a square, but the condition of Theorem \ref{all squares} also requires 
$|S(1)|$ and  $|S(-1)|$ to be squares. 

\medskip Hence it is natural to ask: are there Salem polynomials $S$ such that $|S(1)S(-1)|$ is a square, but 
$|S(1)|$ and  $|S(-1)|$ are not squares ? We thank Chris Smyth for the following example: 

\begin{example}
\label{ex:Smyth}
Set $S(x) = x^{22} - 5x^{21} + 5x^{16} - 5x^{11} + 5x^6 - 5x + 1$. We have $S(1) = -3$, $S(-1) = 27$, hence $S$ satisfies the conditions
of Corollary \ref{one square} but not those of Theorem \ref{all squares}. Therefore there exist K3 surfaces of algebraic dimension 0 with field of  Hodge endomorphisms isomorphic to ${\bf Q}[x]/(S)$, but there is no K3 surface having a Salem automorphism with polynomial $S$. 

\end{example}

This example shows that there exist K3 surfaces of algebraic dimension 0 having Salem multiplication, but no Salem automorphism. 

\medskip

We also have the following result covering all other degrees $d$, except for $d = 10$ and $18$:

\begin{theo} Let $S$ be a Salem polynomial of degree $d$, and assume that $4 \leqslant d \leqslant 20$ and $d \not = 10, 18$. Then there exists a K3 surface of algebraic dimension $0$ having a Salem automorphism with polynomial $S$. 

\end{theo}

\noindent
{\bf Proof.} See \cite{B 25}, Theorem 21.3  and \cite{YT}, Theorem 1.3.
\qed

\medskip
This implies that for degrees $d$ with $d \not = 10, 18$ and $22$, the answer is the same for both questions - all Salem polynomials (and Salem fields) are
realizable.

\bigskip The case $d = 10$ is open in general: a necessary and sufficient condition is given in \cite{B 25}, but it is not known whether it is always satisfied or not.
On the other hand, a necessary and sufficient condition also exists for $d = 18$,
but it is not always satisfied,  as shown by the following example:

\begin{example}\label{second smallest}
Let $\lambda_{18} = 1.1883681475...$, the smallest degree $18$ Salem number (it is also the second smallest known Salem number), and let
$S$ be the corresponding Salem polynomial. There does not exist any K3 surface $X$ with $a(X) = 0$ having an automorphism with dynamical degree
$\lambda_{18}$ (see \cite{B 25},
Example 22.4). On the other hand, Theorem \ref{K3} implies that there exists a K3 surface $X$ of algebraic dimension 0 with $A_X \simeq
{\QQ}[x]/(S)$.

\end{example}

This example also provides us with  K3 surfaces of algebraic dimension 0 having Salem multiplication, but no Salem automorphism.

\subsection{Groups of bimeromorphic automorphisms}\label{Oguiso} 

If $X$ is a HK manifold, we denote by ${\rm Bir}(X)$ its bimeromorphic automorphism group. The following is a result of Oguiso (see \cite{O 8}, Theorem 1.5 (1). 

\begin{theo}\label{Oguiso Bir} If $X$ is elliptic, then we have an exact sequence 
$$1 \to N \to {\rm Bir}(X) \to {\ZZ}^{r}$$
where $N$ is a finite group and $r$ is either 0 or 1. 

\end{theo} 

The aim of this section is to relate the finiteness of the group ${\rm Bir}(X)$ to properties of the algebra of Hodge endomorphisms $A_X$.

\begin{prop} 
\label{prop:<oo}
Let $X$ be an elliptic HK manifold, and suppose that $A_X$ is totally real. Then the group ${\rm Bir}(X)$ is finite.

\end{prop}

\noindent
{\bf Proof.} Let $a$ be an automorphism of $X$, and let $a^* : T_X \to T_X$ be the induced isomorphism. If $a$ is of infinite order, then so is $a^*$,
as we have seen in the proof of Lemma \ref{lem:a*}.
Hence, by \cite[ Theorem 3.4]{O 8}, the characteristic polynomial of $a^*|_{T_X}$ is a Salem polynomial, say $S$.
By Lemma \ref{lem:a*}, this implies that $A_X \simeq
{\QQ}[x]/(S)$.
But this contradicts the hypothesis that $A_X$ is totally real.
\qed


\medskip

For K3 surfaces, the converse of this result also holds

\begin{theo} 
\label{thm:oo}
Let $X$ be a K3 surface of algebraic dimension 0, and suppose that $A_X$ is a Salem field. Then the group  ${\rm Aut}(X)$ is infinite.

\end{theo}

\noindent
{\bf Proof.}
By Proposition \ref{chinburg},
we may write $A_X=E = \QQ(\lambda)$ for a Salem number $\lambda$.
By the proof and in the notation of Theorem \ref{thm:pspolSM},
there is a complex embedding $\sigma: E \to \CC$ such that the 
transcendental Hodge structure $T_\QQ$ on $X$ is determined by
$T^{2,0} = U^0_\sigma$.
Here $\lambda$ acts as multiplication by $\sigma(\lambda)$
and thus preserves the Hodge decomposition.
The theorem thus follows from the following claim:

\begin{claim}
Some power of $\lambda$ is induced from an automorphism of $X$.
\end{claim}

To see this, it suffices by the Torelli theorem to check that some power of $\lambda$ extends to a Hodge isometry
of $H^2(X,\ZZ)$ which preserves the K\"ahler cone. 
But this follows by combining the following properties:
\begin{enumerate}
\item
$T_X$ is a primitive lattice inside $T_{X,\QQ}\cong E$ and thus preserved by some power of $\lambda$, since
$\lambda$ is a unit in $\mathcal O_E$.
\item
$\lambda$ is compatible with the intersection pairing by the adjoint condition \eqref{eq:adj},
since $(\lambda x,\lambda y) = (x, \bar\lambda\lambda y) = (x,y)$ because $\lambda$ is a Salem number,
so $\lambda^{-1} = \bar\lambda$.
\item
Some power of $\lambda$ acts trivially on the discriminant group of $T_X$ (because this is finite abelian)
hence the joint power with (1) glues to the identity acting on $\Pic(X)$ to a Hodge isometry $\phi$ of $H^2(X,\ZZ)$.
\item
The K\"ahler cone is cut out in the positive cone by the hyperplanes orthogonal to the roots in 
$\Pic(X)$,
and these are finite in number because $\Pic(X)$ is negative definite, 
hence the K\"ahler cone is preserved by some power of $\phi$.
\end{enumerate}
By the Torelli theorem, this power of $\phi$ defines an automorphism of $X$ 
which, by construction, has infinite order.
\qed

\smallskip

\begin{coro}[Theorem \ref{thm:bir}]
\label{finite Bir}
A K3 surface $X$ of algebraic dimension {\rm 0}  has finite ${\rm Bir}(X)$ if and only if $A_X$ is totally real. 
\end{coro}

\begin{proof}
This follows by combining Proposition \ref{prop:<oo} and Theorem \ref{thm:oo}
since $\Aut(X) = {\rm Bir}(X)$ for any K3 surface.
\qed
\end{proof}

\begin{remark}
We will discuss the general HK case in future work,
thus also extending results like Theorem \ref{all squares}
in that direction.
\end{remark}

\subsection{Algebraic entropy and Salem fields}\label{entropy}

Let $X$ be a HK manifold, and let $a \in {\rm Bir}(X)$ be a bimeromorphic automorphism. The {\it algebraic entropy}  
of $a$ is defined as the logarithm
of the dynamical degree of $a$.
For different notions of entropy and their interrelations, see  Gromov \cite{Gromov}, Yomdin \cite{Y} and  \cite{O 7}. 
With the notation of \S \ref{Stark and Chinburg} we have

\begin{theo} 
\label{thm:entropies}
Let $E$ be a Salem field of degree $2d$, and let $X$ be an elliptic HK manifold with Salem multiplication by $E$. Then there exists an integer
$n_X \geqslant 0$ such that  the set of algebraic entropies
of $X$ consists of the positive integral multiples of 

$$n_XL'(0,\chi)(h(E_0)/h(E))2^{1-d}.$$

\end{theo} 

\noindent
{\bf Proof.} If ${\rm Bir}(X)$ is finite, set $n_X = 0$. 
Suppose that ${\rm Bir}(X)$ is infinite, and let $a \in {\rm Bir}(X)$ be a generator of the infinite
cyclic subgroup of ${\rm Bir}(X)$ from Theorem \ref{Oguiso Bir}. 
Since ${\rm Bir}(X)$ acts through a finite group on the negative-definite lattice $\Pic(X)$,
the characteristic polynomial $\chi$ of $a^*|_{T_X}$ is a Salem polynomial as in the proof of Lemma \ref{lem:a*}.
Denoting the corresponding Salem number by $\lambda$ and the Salem field by $E'$,
we have $E'\subset A_X$, so $E\cong E'$ by Proposition \ref{prop:E=F}.
By Theorem \ref{thm:eps}, with $u\in\{1,2\}$ as defined there,
we find that  $\lambda$ is a power of $\sqrt\epsilon = {\rm exp}(h(E_0)2^{1-d}u L'(0,\chi)/h(E))$. Let us write 
$\lambda = (\sqrt \epsilon)^{m_X}$, for some integer $m_X \geqslant 1$, and set $n_X = m_Xu$. Then the algebraic entropy
of $a$ is equal to $n_XL'(0,\chi)(h(E_0)/h(E))2^{1-d}$, and this completes the proof of the theorem. 
\qed

\medskip

\subsection*{Acknowledgements}

We are grateful to the referee for many helpful comments and corrections.
We thank Ted Chinburg  and Chris Smyth for very interesting discussions, 
especially for their help concerning Proposition \ref{chinburg}. 
We also thank Chris Smyth for providing the Salem polynomial
which Example \ref{ex:Smyth} builds on.

%
%
%
%
%
%
%
%
%
%

\end{document}